\documentclass[12pt]{article}
\usepackage{amssymb, latexsym, amsthm, amsmath, amsfonts}
\usepackage{mathrsfs}
\usepackage{geometry}

\newtheorem{theorem}{Theorem}
\newtheorem*{theorem*}{Theorem}
\newtheorem{lemma}{Lemma}[section]
\newtheorem*{lemma*}{Lemma}
\newtheorem{corollary}[lemma]{Corollary}
\newtheorem*{corollary*}{Corollary}
\newtheorem{proposition}[lemma]{Proposition}
\newtheorem*{proposition*}{Proposition}
\newtheorem*{example*}{Example}
\newtheorem{question}{Question}
\newtheorem*{question*}{Question}

\newcommand{\fD}{\mathfrak{D}}
\newcommand{\nat}{\mathbb{N}}
\newcommand{\sR}{\mathscr{R}}
\newcommand{\sP}{\mathscr{P}}
\newcommand{\fL}{\mathfrak{L}}
\newcommand{\cC}{\mathcal{C}}
\newcommand{\cF}{\mathcal{F}}

\newcommand{\PA}{\mathsf{PA}}
\newcommand{\EA}{\mathsf{EA}}
\newcommand{\PRA}{\mathsf{PRA}}
\newcommand{\ISi}{\mathsf{I}\Sigma}

\newcommand{\Rfn}{\mathsf{Rfn}}
\newcommand{\RFN}{\mathsf{RFN}}
\newcommand{\Prf}{\mathsf{Prf}}

\newcommand{\ul}{\ulcorner}
\newcommand{\ur}{\urcorner}
\newcommand{\la}{\langle}
\newcommand{\ra}{\rangle}

\newcommand{\eqv}{\leftrightarrow}
\newcommand{\dom}{\mathrm{dom}}

\newcommand{\gn}[1]{\ulcorner #1 \urcorner}
\renewcommand{\phi}{\varphi}
\renewcommand{\sharp}{\mathrm{\#}}
\newcommand{\num}{\underline}

\title{On Shavrukov's non-isomorphism theorem for diagonalizable algebras}

\author{Evgeny A.\,Kolmakov\\%\footnote{\texttt{kolmakov-ea@yandex.ru}}\\
Steklov Mathematical Institute of Russian Academy of Sciences\\ Moscow, Russia\\

}

\begin{document}

\maketitle

\begin{abstract}
	We prove a strengthened version of V.\,Yu.\,Shavrukov's result on the non-isomorphism of diagonalizable algebras of two $\Sigma_1$-sound theories, based on the improvements previously found by G.\,Adamsson. We then obtain several corollaries to the strengthened result by applying it to various pairs of theories and obtain new non-isomorphism examples. In particular, we show that there is epimorphism from $(\fL_T, \Box_T\Box_T)$ onto $(\fL_T, \Box_T)$. The case of bimodal diagonalizable algebras is also considered. We give several examples of pairs of theories with isomorphic diagonalizable algebras but non-isomorphic bimodal diagonalizable algebras.
\end{abstract}

\section{Introduction}

The equational class of diagonalizable algebras (also known as Magari algebras) was introduced by Magari \cite{Mag75} to study the notion of provability in a formal theory within a general algebraic setting. A \emph{diagonalizable algebra} is a boolean algebra $(\mathcal{B}, \Box)$ together with a unary operator $\Box$ satisfying the following identities:
\begin{enumerate}
	\item $\Box(a \to b) \to (\Box a \to \Box b) = \top$,
	\item $\Box(\Box a \to a) \to \Box a = \top$, 
	\item $\Box\top = \top$.
\end{enumerate}
Given a formal theory $T$ (containing enough arithmetic to be able to reason about its own syntax and, in particular, provability) one can associate with it a specific diagonalizable algebra $\fD_T$ called the \emph{diagonalizable algebra of} $T$ (the term ``\emph{provability algebra}'' is also widely used for such algebras, but we stick with the former following Shavrukov). The construction is as follows. Let $\fL_T$ be the \emph{Lindenbaum boolean algebra of} $T$. The universe of $\fL_T$ is the set of all equivalence classes $[\phi]_{\sim_T}$ of $T$-sentences, where the equivalence relation $\phi \sim_T \psi$ is a $T$-provable equivalence, i.e., $\phi \sim_T \psi$ if and only if $T \vdash \phi \eqv \psi$. The boolean operations on the equivalence classes are induced  by propositional connectives. The algebra $\fD_T$ is then defined as $(\fL_T, \Box_T)$, where $\Box_T$ is a unary operator induced by the \emph{provability predicate $\mathsf{Pr}_T(x)$ of} $T$, a formula in the language of $T$ which expresses the fact that a formula with g{\"o}delnumber $x$ is provable in $T$, as follows $\Box_T[\phi]_{\sim_T} = [\mathsf{Pr}_T(\gn{\phi})]_{\sim_T}$. The formula $\mathsf{Pr}_T(x)$ should express provability in $T$ in such a way that $\fD_T$ would indeed satisfy the identities above. This is true, e.g., in case when $\mathsf{Pr}_T(x)$ is the \emph{standard provability predicate for} $T$, i.e., has the form $\exists p\, \Prf_T(p, x)$, where $\Prf_T(p, x)$ is the \emph{standard proof predicate for} (a fixed axiomatization of) $T$, expressing in a natural way that $p$ codes a Hilbert-style proof of a formula with g{\"o}delnumber $x$ from the axioms of $T$. Note that the identities 1, 2 and 3 above then correspond to provable closure of the set of theorems of $T$ under modus ponens, formalized L{\"o}b's theorem and the fact that $T$ can prove the provability of its theorems.
One may also consider the operators induced by the non-standard provability predicates as long as the obtained algebra $\fD_T$ satisfies the identities defining diagonalizable algebras. The above discussion implies that the notation $\fD_T$ is ambiguous, since $\fD_T$ depends on the choice of a formula $\mathsf{Pr}_T(x)$, which itself may depend on a particular axiomatization of~$T$. Nevertheless, we use this notation always assuming that the axiomatization and a formula $\mathsf{Pr}_T(x)$ for a theory $T$ in question are fixed.

The formal study of provability as an operator (or a modality) is traditionally conducted within the field, known as provability logic.
The diagonalizable algebras of formal theories provide a different way of looking at this field.
One of the central results in this area is the famous Solovay arithmetical completeness theorem \cite{Sol76}, which can be reformulated in the algebraic setting as follows: for each $\Sigma_1$-sound theory $T$ the equational theory (or rather the logic) of $\fD_T$ coincides with propositional modal logic $\mathsf{GL}$, called the G{\"o}del-L{\"o}b logic. The uniform version of this theorem is translated into the existence of an embedding of the countably generated free diagonalizable algebra into $\fD_T$.

Several of the most prominent results on diagonalizable algebras of formal theories were obtained by V.\,Yu.\,Shavrukov (see \cite{Shav93a, Shav93b, Shav97a, Shav97b}). Among them the nonarithmeticity of the full first order theory $\mathsf{Th}(\fD_T)$ of the diagonalizable algebra of a $\Sigma_1$-sound theory $T$ (and the first order theory of any collection of diagonalizable algebras of such theories) and the undecidability of its $\forall^* \exists^* \forall^* \exists^*$-fragment, a complete description of r.e.\,subalgebras of $\fD_T$ (later D.~Zambella extended this result to arbitrary subalgebras), the existence of non-trivial automorphisms of~$\fD_T$. 
Shavrukov have also obtained a couple of important results on isomorphisms of the diagonalizable algebras of formal theories (see \cite{Shav93a, Shav97a}). It is these results, which we formulate below, that we focus on in the present work. For the precise definitions
of the notions and notation used in the statements see the next section.

The first one is a general sufficient condition for the diagonalizable algebras of two $\Sigma_1$-sound theories to be not elementarily equivalent, whence also not isomorphic (see \cite{Shav93a} and \cite[Theorem 2.11]{Shav97b}). In particular, it implies that $\fD_\PA$ and $\fD_{\mathsf{ZF}}$ are not isomorphic. 

\begin{theorem}\label{th:shavnoniso}
	If $T$ proves the following version of the uniform $\Sigma_1$-reflection for $S$ 
	$$
	\forall x\, \exists y\, \forall \delta(\cdot) \in \Delta_0 
	\left( \exists p \leqslant x\, \Prf_S(p, \gn{\exists z\, \delta(z)}) \to \exists u \leqslant y\, \mathsf{Tr}_{\Delta_0}(\gn{\delta(\num{u})}),
	\right),
	$$ then $\fD_T \not \cong \fD_S$. Moreover, in this case $\fD_T$ and $\fD_S$ are not elementarily equivalent. 
\end{theorem}
This may be contrasted with a theorem of Pour-El and Kripke \cite{KriPou67} stating that the Lindenbaum algebras $\fL_T$ and $\fL_S$ are recursively isomorphic (for all reasonable theories $T$ and~$S$).

The second important result of Shavrukov is a sufficient condition for the (recursive) isomorphism of $\fD_T$ and $\fD_S$, from which it follows that, e.g., $\fD_{\mathsf{ZF}} \cong \fD_{\mathsf{GB}}$ (see \cite{Shav97a}). Here we formulate it in a slightly weakened form, which is sufficient for all applications. Given a class of formulas~$\Gamma$, 
two theories $T$ and $S$ are called \emph{$\Gamma$-coherent} if
\begin{itemize}
	\item[$(i)$] $T \vdash \gamma \Longleftrightarrow S \vdash \gamma$ for each sentence $\gamma \in \Gamma$,
	
	\item[$(ii)$] both $T$ and $S$ prove the following ``uniform'' formalization of point $(i)$,
	\begin{align*}
	\forall x\, \exists y\, \forall \gamma \in \Gamma &\Bigl(
	\left(\exists p \leqslant x\, \Prf_T(p, \gn{\gamma}) \to 
	\exists q \leqslant y\, \Prf_S(q, \gn{\gamma}) \right) \land \Bigr.\\
	& \quad \Bigl. \left(\exists p \leqslant x\, \Prf_S(p, \gn{\gamma}) \to 
	\exists q \leqslant y\, \Prf_T(q, \gn{\gamma}) \right)
	\Bigr).
	\end{align*}
\end{itemize}

\begin{theorem}\label{th:shaviso}	
	If $T$ and $S$ are $\mathcal{B}(\Sigma_1)$-coherent theories, then 
	$\fD_T \cong \fD_S$.
\end{theorem}

Later G.\,Adamsson obtained the following theorem \cite[Theorem 17]{Adam11}, 
which
improved Shav\-ru\-kov's non-isomorphism condition in two ways.  
Here $\sP_T(x)$ is the \emph{Parikh speed-up function for} $T$ (see \cite[Definition 5]{Adam11}) defined as follows: the value $\sP_T(n)$ is the least upper bound on the length (number of symbols) of $T$-proofs of all formulas $\psi$ such that $\Box_T\psi$ has a $T$-proof with at most $n$ symbols.

\begin{theorem}\label{th:adamnoniso}
	If there is an epimorphism from $\fD_T$ onto $\fD_S$, then
	there is an elementary function $t(x)$ such that 
	$\sP_T(n) \leqslant t(\sP^{(4)}_S(n))$
	holds for infinitely many $n$.
\end{theorem}

 Firstly, he weakened the condition itself, which allowed to obtain new non-isomorphism results demonstrating that the isomorphism class of $\fD_T$ actually depends on the choice of the formula $\mathsf{Pr}_T(x)$. In particular, he showed \cite[Corollary 21]{Adam11} that $(\fL_T, \Box_T) \not \cong (\fL_T, \Box^{(6)}_T)$, where $\Box^{(6)}_T$ is, of course, $\Box_T$ iterated six times.
  Secondly, he showed that this weakened condition is sufficient to obtain the non-existence of an epimorphism from $\fD_T$ onto $\fD_S$, rather than just isomorphism. In his work Adamsson notes that the obtained condition is ``by no means optimal'' and ``sharper bounds are certainly possible''.

Our work is a natural continuation of this line of research. The main theorem we prove (Theorem \ref{th:main}) provides a more general version of Adamsson's condition, demonstrating that sharper bounds are indeed possible. 
To do so we first discuss several properties of \emph{the reflection function $\sR_T(x)$ for} $T$ (which is closely related to $\sP_T(x)$ present in Adamsson's condition), which are then used to obtain the desired improvements in the main theorem, and introduce the classes of functions $\cF^k_{T}(U)$, which can be seen as the generalizations of the class of provably total computable functions $\cF(T)$, and are implicitly present in Adamsson's approach.
Using the improved non-isomorphism condition we derive that $(\fL_T, \Box_T) \not \cong (\fL_T, \Box_T\Box_T)$, a result anticipated in \cite{Adam11}, and obtain a more general non-isomorphism of this kind. More importantly, we consider several more natural examples of theories with non-isomorphic algebras but the same classes of provably total computable functions, among them two $\Pi_1$-coherent theories.
Finally, we consider the \emph{bimodal diagonalizable algebras}, where the second modality is interpreted as an operator of \emph{1-provability}, i.e., provability in a theory together with all true $\Pi_1$-sentences added as additional axioms. We naturally focus on pairs of theories with isomorphic diagonalizable algebras, but non-isomorphic bimodal algebras. We start with examples, where non-isomorphism can be established more algebraically (as opposed to Shavrukov's approach), and also obtain a family of similar examples by applying our improved version of the Shavrukov--Adamsson theorem.

The paper is organized as follows. 
Section 2 introduces the basic notions and
notation used throughout the paper. 
In Section 3 we define the classes $\cF^k_{T}(U)$, study their properties and prove certain important facts about the reflection function $\sR_T(x)$. 
Section 4 contains the main results of the paper, namely, the main theorem on the non-isomorphism of diagonalizable algebras and its various applications to natural pairs of theories. In Section 5 we discuss certain non-isomorphism results on bimodal diagonalizable algebras.

\section{Preliminaries}

In this paper we consider first-order theories that contain some basic arithmetical theory. We always assume that if a theory in question is not formulated in the language of arithmetic, then some standard translation of arithmetical language into the language of the theory is fixed, and we identify arithmetical formulas with their translations.
As our basic arithmetical theory we take \emph{Elementary arithmetic} $\EA$, also known as $\mathsf{I}\Delta_0(\exp)$, that is, the first-order theory formulated in the pure arithmetical language $0, (\cdot)', +, \times$ enriched with the unary function symbol $\exp$ for the exponentiation function $2^x$.
It has the standard defining axioms for these symbols and the induction schema for all \emph{elementary formulas} (or \emph{$\Delta_0(\exp)$-formulas}), 
i.e., formulas in the language with exponent containing only bounded (by the terms in the language with exponent) quantifiers. Along with it we also consider a smaller class of \emph{$\Delta_0$-formulas}, which are the formulas in pure arithmetical language (without exponent) containing only bounded quantifiers.
We define the classes $\Sigma_0$ and $\Pi_0$ to be $\Delta_0$. 
The classes $\Sigma_n$ and $\Pi_n$ of \emph{arithmetical hierarchy} are then defined recursively in a standard way for all $n > 0$. We also consider the classes of boolean combinations of $\Sigma_n$-formulas which are denoted by  $\mathcal{B}(\Sigma_n)$.

The theory $\EA^+$ is defined to be an extension of $\EA$ with the axiom asserting the totality of the superexponential function $2^x_y$ (which can be naturally expressed by a $\Pi_2$-formula in the language of $\EA$).
If we allow induction for all arithmetical formulas, the resulting theory is \emph{Peano arithmetic} denoted by $\PA$.
For a fixed class of arithmetical formulas $\Gamma$ the fragment of $\PA$ obtained by restricting the \emph{induction schema}
$$
\forall y \left(\phi(0, y) \land \forall x \left(\phi(x, y) \to \phi(x + 1, y)\right) \to \forall x\, \phi(x, y)\right), \quad \phi(x, y) \in \Gamma
$$
to $\Gamma$-formulas  is denoted by $\mathsf{I}\Gamma$.
We also consider the following \emph{schema of $\Sigma_1$-collection}, denoted by  $\mathsf{B}\Sigma_1$,
$$
\forall x < z\, \exists y\, \phi(x, y, a) \to \exists u\, \forall x < z\, \exists y < u\, \phi(x, y, a), \quad \phi(x, y, a) \in \Sigma_1.
$$

We assume that some standard arithmetization of syntax and the g{\"o}delnumbering of syntactic objects is fixed.
In particular, we write $\gn{\phi}$ for the (numeral of the) \emph{g{\"o}delnumber of}~$\phi$.
All the theories appearing throughout the paper are assumed to be \emph{recursively axiomatizable $\Sigma_1$-sound extensions} of $\EA$.
Recall that \emph{$\Sigma_1$-soundness} means that each provable $\Sigma_1$-sentence is true in the standard model $\nat$.
Moreover, (unless explicitly stated otherwise) we also assume that each theory $T$ is given to us by its \emph{elementary axiomatization}, i.e., an elementary formula $\sigma_T(x)$,
defining the set of axioms of $T$ in the standard model. 
This formula $\sigma_T$ is used in the construction of
the \emph{standard proof predicate} $\Prf_{\sigma_T}(p, x)$ for $T$, 
an elementary formula expressing the relation ``$p$ codes a Hilbert-style proof of the formula with g{\"o}delnumber $x$ from the axiom set defined by the formula $\sigma_T$''. 
%Note that (for a natural definition) the proof predicate $\Prf_{\sigma_T}(p, y)$ contains $\sigma_T$ as a subformula.
The \emph{standard provability predicate} for $T$ (w.r.t. a fixed elementary axiomatization $\sigma_T$) is then given by $\exists p\, \Prf_{\sigma_T}(p, x)$, and we denote this formula by $\Box_{\sigma_T}(x)$.
In practice, almost everywhere we use the following less detailed notation $\Prf_T(p, x)$ and $\Box_T(x)$ to denote $\Prf_{\sigma_T}(p, x)$ and $\Box_{\sigma_T}(x)$ for some fixed axiomatization $\sigma_T$ of a theory $T$ under consideration, which does not lead to confusion if we consider a single axiomatization. But if we are dealing with several axiomatizations of the same theory we revert back to the official notation to make the difference between the proof and provability predicates for different axiomatizations explicit.
We use the notation $T \vdash_p \phi$ for $\exists q \leqslant p\, \Prf_T(q, \gn{\phi})$.
As usual, we also write $\Box_T \phi$ instead of $\Box_T(\gn{\phi})$ and quantify over (the g{\"o}delnumbers of) formulas, their sequences and other syntactic objects freely. The finite iterations of $\Box_T$ are denoted by $\Box^{(k)}_T$.

The standard provability predicate $\Box_T$ satisfies the following \emph{L{\"o}b's derivability conditions}:
\begin{enumerate}
	\item If $T \vdash \varphi$, then $\EA \vdash \Box_T\varphi$,
	\item $\EA \vdash \Box_T(\varphi \to \psi) \to (\Box_T\varphi \to \Box_T\psi)$,
	\item $\EA \vdash \Box_T\varphi \to \Box_T\Box_T\varphi$.
\end{enumerate}
Moreover these conditions hold provably in $\EA$ with $\phi$ and $\psi$ being universally quantified over.
Point 3 follows from the general fact known as \emph{provable $\Sigma_1$-completeness}:
$$
\EA \vdash \forall x_1\, \dots \forall x_m \left(\sigma(x_1, \dots, x_m) \to \Box_T\gn{\sigma(\num{x}_1, \dots, \num{x}_m)}\right),
$$
whenever $\sigma(x_1, \dots, x_m)$ is a $\Sigma_1$-formula. 
Here the underline notation $\ul \varphi(\num{x}) \ur$ stands for the elementarily definable term, representing the elementary function
that maps $k$ to the g{\"o}delnumber of the formula $\ul\varphi(\num{k})\ur$, where $\num{k}$ is the $k$th numeral, i.e., the term $0^{' \dots '}$ with $k$ primes.
In what follows we usually write just $\Box_T \phi(\num{x})$ instead of $\Box_T \gn{\phi(\num{x})}$.

We say that two theories $T$ and $S$ are \emph{$\Gamma$-equivalent} if 
$T \vdash \gamma$ if and only if $S \vdash \gamma$ for each sentence $\gamma \in \Gamma$. In case $\Gamma$ is the class of all arithmetical sentences we say that $T$ and $S$ are \emph{deductively equivalent} and use the notation $T \equiv S$.

Among the extensions of a given theory $T$ we consider the \emph{reflection principles}. The \emph{local $\Gamma$-reflection principle} $\Rfn_{\Gamma}(T)$ is the schema $\Box_T \gamma \to \gamma$ for each sentence $\gamma \in \Gamma$. Similarly the \emph{uniform $\Gamma$-reflection principle} $\RFN_{\Gamma}(T)$ is the schema $\forall x \left(\Box_T \gamma(\num{x}) \to \gamma(x) \right)$ for each formula $\gamma \in \Gamma$. Usually $\Gamma$ is one of the classes $\Sigma_n$ or $\Pi_n$ for $n > 0$. If $\Gamma$ is the set of all arithmetical formulas we write $\Rfn(T)$ and $\RFN(T)$.
Let us also mention the following useful principle
known as the \textit{small reflection},
for each formula $\delta(x)$ we have
$$
	\EA \vdash \forall x\, \forall y\, \Box_T \left(\Prf_T(\num{y}, \gn{\delta(\num{x})}) \to \delta(\num{x})\right).
$$

It is known that for each $n > 0$ there exists a $\Pi_n$-formula 
$\mathsf{Tr}_{\Pi_n}(x)$, the \emph{truth definition} for the class of $\Pi_n$-formulas (and similarly for $\Sigma_n$-formulas), such that
$$
\EA \vdash \forall x_1,\dots, x_m\:(\varphi(x_1, \dots, x_m) \eqv \mathsf{Tr}_{\Pi_n}(\gn{\varphi(\num{x}_1, \dots, \num{x}_m)})),
$$
for every $\Pi_n$-formula $\varphi(x_1, \dots, x_m)$,
and this fact itself is formalizable in $\EA$, so
$$
\EA \vdash \forall \pi \in \Pi_n\, \Box_T \left(\pi \eqv \mathsf{Tr}_{\Pi_n}(\pi)\right).
$$
The same also holds for $n = 0$, but note that the truth predicate for $\Delta_0$-formulas $\mathsf{Tr}_{\Delta_0}(x)$ used in the construction of the above truth predicates is defined by a $\Delta_0(\exp)$-formula, rather than a $\Delta_0$-formula.
The truth predicates are used to define the so-called \emph{strong provability predicates}. In this paper we only consider the predicate of \emph{$1$-provability},
i.e., usual provability in $T$ together with all true $\Pi_1$-sentences taken as additional axioms. Namely, we define the predicate of \emph{$1$-provability} in $T$ by the following formula
$$
[1]_T \varphi := \exists \pi \left(\mathsf{Tr}_{\Pi_1}(\pi) \wedge \Box_T(\pi \to \phi)\right),
$$
The predicate $[1]_T$ satisfies the same derivability conditions as $\Box_T$ and is provably $\Sigma_{2}$-complete.
It is not hard to see that $\RFN_{\Sigma_1}(T)$ is equivalent to $\neg [1]_T \bot$ provably in $\EA$. 
We also recall the following notation from \cite{KolmBek19} for the set of sentences which are \emph{$T$-provably $1$-provable in} $S$,
$$
C_S(T) := \{\psi \mid T \vdash [1]_S\psi \}.
$$

A function $f \colon \nat^k \to \nat$ is said to be a \emph{provably total computable function of} $T$ if there is a $\Sigma_1$-formula $\psi(x_1, \dots, x_k, y)$ defining the graph of $f$ in the standard model and 
$$
T \vdash \forall x_1\, \dots \forall x_k\, \exists y\, \psi(x_1, \dots, x_k, y).
$$ 
The \emph{class of provably total computable functions of} $T$ is denoted by $\cF(T)$. It is known that $\cF(\EA) = \mathcal{E}_3$, the class of \emph{Kalmar elementary functions}. Since we deal only with extensions of $\EA$, the class $\cF(T)$ always contains all elementary functions. When comparing the growth rate of functions in this paper, we do it up to an elementary transformation of the input, and use the following notation for (partial) functions $f, g \colon \nat \to \nat$ and $X \subseteq \nat$. We write
$f \leqslant_X g$, if $X \subseteq \dom\, f \cap \dom\, g$ and for each $n \in X$ we have $f(n) \leqslant g(n)$, and
$f \preccurlyeq_X g$, if there is an elementary function $q(x)$ such that $f \leqslant_X q \circ g$. We also write $f \approx_X g$, when both $f \preccurlyeq_X g$ and $g \preccurlyeq_X f$ hold. The notation $f =_X g$ means that for each $n \in X$ we have that either both $f(n)$ and $g(n)$ are undefined, or both values are defined and equal. In case $X = \nat$ we just write $f = g$.
We assume that some standard g{\"o}delnumbering of Turing machines and their computations is fixed. In particular, the \emph{universal function} is denoted by $\phi_e(x)$, so $\phi_e$ is the unary function, computed by the Turing machine with code $e$. The number $e$ in this case is called a \emph{$\phi$-index} of $f$ (other indexing functions will be introduced later). We also consider the following \emph{complexity measure} associated with $\phi_e$,
$$
\Phi_e(x) := x + e + \text{the number of steps in the computation $\phi_e(x)$},
$$ 
whenever the computation $\phi_e(x)$ converges. We use these expressions in formalized contexts and assume that the formalization is natural enough, so that the Kleene $T$-predicate and, in particular, the predicate $\Phi_e(x) \leqslant y$ are expressed by the elementary formulas. We also say that $f(x)$ \emph{is dominated by} $g(x)$ if there is a $m \in \nat$ such that for all $n \geqslant m$ we have $f(n) < g(n)$. The expression $f^{(n)}(x)$ denotes the $n$-fold composition of $f(x)$ with itself, and $f^{(0)}(x) = x$.
Finally, we say that $f(x)$ is \emph{elementarily cumulative} (or just \emph{cumulative}, for short) if there is a $\phi$-index $a$ for $f$ such that $\Phi_a \preccurlyeq_{\dom\, f} f$, so the value $f(x)$ can be computed in time elementary in it. This notion is very close to the notion of elementarily honest function introduced in \cite{MeRi72}  (but it is not the same, since, for instance, all elementary $\{0,1\}$-valued functions are honest yet not cumulative), which turns out to be equivalent to $f$ having an elementary graph. We, however, stick with the notion of cumulativity used in \cite{Shav93a}, one of the reasons being that in some modern definitions of honesty, e.g., used in the study of honest elementary degrees  (see \cite{Kris98}), several additional properties, like being strictly increasing or growing at least as fast as the function~$2^x$, are required.

\section{Reflection function $\sR_T(x)$ and the classes $\cF^k_{T}(U)$}

In his work Adamsson introduces several functions (see \cite[Definition 5]{Adam11}) measuring the proof speed-up obtained by adding a certain inference rule to $T$, e.g., the Parikh rule (from $\Box_T\psi$ derive $\psi$), 
and shows them to be closely related in terms of the growth rate.
Adamsson's necessary condition (see Theorem \ref{th:adamnoniso}) for the existence of an epimorphism from $\fD_T$ onto $\fD_S$ is formulated as a restriction on the growth rate of one of such functions for the theory $T$ in terms of the same function for the theory $S$. Namely, it asserts that the Parikh speed-up function $\sP_T(x)$ for $T$ cannot dominate all compositions of the form $t(\sP^{(4)}_S(x))$, where $t(x)$ is an elementary function and $\sP_S(x)$ is the Parikh speed-up function for $S$.
In this section we obtain a stronger condition of this form.

As a measure of proof size Adamsson uses the number of symbols. In our presentation we (following Shavrukov \cite{Shav93a}) use the g{\"o}delnumber of a proof instead. This does not lead to drastic changes in the growth rate since these two measures are elementarily related. Also, instead of the Parikh speed-up function $\sP_T(x)$ we consider the $\Sigma_1$-reflection function $\sR_T(x)$, which is defined below, for the sake of it being more general and easier to work with. But since these two functions are elementarily related (see \cite[Proposition 5]{Adam11}), the results can also be reformulated in terms of $\sP_T(x)$.

Define the \emph{$\Sigma_1$-reflection function} $\sR_T(x)$ for a theory $T$ as follows
$$
\sR_T(x) = \mu y.\left[
\forall \delta(\cdot) \in \Delta_0 \left(
T \vdash_x \exists z\, \delta(z) \to \exists u \leqslant y\, \delta(u)
\right)
\right],
$$
i.e., $\sR_T(x)$ gives the least upper bound to the least witnesses of all $\Sigma_1$-sentences, which have a $T$-proof with g{\"o}delnumber less than or equal to $x$ (cf. the reflection principle in Theorem \ref{th:shavnoniso}). Since $T$ is $\Sigma_1$-sound, the function $\sR_T(x)$ is recursive.
As can be seen from the definition, it certainly depends on a particular axiomatization of $T$, since it is defined in terms of a proof predicate for $T$. So an appropriate notation would rather be $\sR_{\sigma_T}(x)$, where $\sigma_T(z)$ is an elementary formula defining the set of axioms of $T$. However, as in the case with the provability predicate $\Box_T$, we usually write just $\sR_T(x)$, since $\sigma_T(z)$ is assumed to be fixed in each particular case. Let us also note the use of $\Delta_0$-formulas (rather than $\Delta_0(\exp)$) in the definition of $\sR_T(x)$, as well as in the definition of the classes of arithmetical hierarchy, which is related to the lack of an elementary truth predicate for the class of $\Delta_0(\exp)$-formulas.

A typical application of $\sR_T(x)$ is when $\exists z\, \delta(z)$ is taken to be a $\Sigma_1$-sentence ($\EA$-provably equivalent to) $\Box_T \psi$, i.e., $\exists p\, \Prf_T(p, \gn{\psi})$. We then can find an upper bound on a (g{\"o}delnumber of a) $T$-proof of $\psi$, given such bound on a $T$-proof of $\Box_T \psi$. Namely, if $T \vdash_x \Box_T \psi$, then 
$T \vdash_{\sR_T(x)} \psi$. Note that although $\Prf_T(p, \gn{\psi})$ is an elementary formula (rather than just $\Delta_0$), we still have the above implication since we assume $\exists p\, \Prf_T(p, \gn{\psi})$ to be written in the form $\exists z\, \delta(z)$, where $\delta(z)$ is a $\Delta_0$-formula, in the most natural way (e.g., replacing each occurrence of $2^x$ with its $\Delta_0$-definition and bounding all existential quantifiers with a single one $\exists z$ in front). The least witness to $\exists z\, \delta(z)$ is then certainly not less than the least witness to $\exists p\, \Prf_T(p, \gn{\psi})$, as required. Similar applications of $\sR_T(x)$ in case $T \vdash_x \Box_T\psi$ with $\mathsf{Prf}_T(p, x)$ being an elementary formula rather than $\Delta_0$ are also present in \cite[Lemma 3(a)]{Adam11} and \cite[Lemma 4]{Shav93a}.

As noted in \cite{Adam11} the function $\sR_T(x)$ is close to being a $T$-provably total computable function (see Proposition \ref{prop:boxkRk}), although $\sR_T(x) \notin \cF(T)$. It is this observation that Adamsson uses to improve Shavrukov's non-isomorphism condition.
The following definition of the \emph{classes} $\cF^k_{T}(U)$ is inspired by this observation.
Denote by $\cF_{T}(U)$ the class of all functions $f \colon \nat^n \to \nat$ such that 
$$
U \vdash \forall x_1 \dots \forall x_n \Box_T \left(\exists y\, \psi_f(\num{x_1}, \dots, \num{x_n}, y)\right),
$$
for some $\Sigma_1$-formula $\psi_f(x_1, \dots, x_n, y)$ defining the graph of $f(x_1, \dots, x_n)$ in the standard model. As usual, the appropriate notation would be $\cF_{\Box_T}(U)$, since the class depends on the provability predicate $\Box_T$ of $T$. However, as $\Box_T$ is usually assumed to be fixed, we use a less heavy notation $\cF_T(U)$ instead.
Note that the definition makes sense for an arbitrary provability predicate instead of $\Box_T$ (not necessarily standard). In particular, we can consider a \emph{sequence of classes} $\cF^k_{T}(U)$ for $k \geqslant 1$, where we use $\Box^{(k)}_T$ instead of $\Box_T$ above. We also abuse the notation and sometimes write $\Box^{(0)}_T \psi(\num{x})$ instead of $\psi(x)$, 
when no confusion arises,
although $\Box^{(0)}_T(x)$ is not a formula. It is then natural to define  $\cF^0_{T}(U)$ to be $\cF(U)$, the class of all provably total computable functions of $U$.
 By provable $\Sigma_1$-completeness we have
$$
\cF(U) = \cF^0_{T}(U) \subseteq \cF_{T}(U) = \cF^1_{T}(U) \subseteq \cF^2_{T}(U) \subseteq \dots \subseteq \cF(U + \RFN_{\Sigma_1}(T)),
$$
and also $\cF(T) \subseteq \cF_{T}(\EA)$, since $\EA \vdash \Box_T\forall x\, \psi(x) \to \forall x\, \Box_T \psi(\num{x})$.
Thus, we get an increasing sequence of classes which are not, in general, closed with respect to composition (unlike $\cF(U)$). 
We, however, have the following closure property.
\begin{lemma}\label{lm:Fkcompos}
	If $T$ contains $U$, then for each $F(x) \in \cF^m_{T}(U)$ and $G(x) \in \cF^n_{T}(U)$ we have
	$G(F(x)) \in \cF^{m + n}_{T}(U)$.
\end{lemma}
\begin{proof}
We consider the case $m > 0$ and $n > 0$. Fix $\Sigma_1$-formulas $\psi_F(x, y)$ and $\theta_G(x, y)$ defining the graphs of $F(x)$ and $G(x)$ such that
$$
U \vdash \forall x\, \Box^{(m)}_T \exists y\, \psi_F(\num{x}, y) \quad 
\text { and } \quad 
U \vdash \forall x\, \Box^{(n)}_T \exists y\, \theta_G(\num{x}, y).
$$
Consider a $\Sigma_1$-formula $\chi(x, z) \leftrightharpoons 
\exists y \left(
\psi_F(x, y) \land \theta_G(y, z)\right)$, which defines the graph of $G(F(x))$. We prove that 
$
U \vdash \forall x\, \Box^{(m + n)}_T \exists z\, \chi(\num{x}, z).
$
Indeed, by provable $\Sigma_1$-completeness we obtain
\begin{align*}
U \vdash \psi_F(x, y) &\to \Box^{(n)}_T \left(\psi(\num{x}, \num{y}) \land \exists z\, \theta_G(\num{y}, z)\right)\\
&\to \Box^{(n)}_T \exists z\, \chi(\num{x}, z),
\end{align*}
whence $U \vdash \forall x  \left(\exists y\, \psi_F(x, y) \to \Box^{(n)}_T \exists z\, \chi(\num{x}, z) \right)$. Now, since $T$ contains $U$, the same is provable in $T$, and therefore ($m > 0$) we have
$$
\EA \vdash \Box^{(m)}_T \forall x  \left(\exists y\, \psi_F(x, y) \to \Box^{(n)}_T \exists z\, \chi(\num{x}, z) \right),
$$
whence we obtain
$$
\EA \vdash \forall x\, \Box^{(m)}_T \exists y\, \psi_F(\num{x}, y) \to \forall x\, \Box^{(m + n)}_T \exists z\, \chi(\num{x}, z).
$$
By the assumption the antecedent of this implication is provable in $U$, whence its consequent is provable in $U$, as required.
The cases $m = 0$ and $n = 0$ can be considered similarly.

\end{proof}

The following proposition confirms the above claim that $\sR_T(x)$ is close to being a $T$-provably total computable function.
We closely follow Adamsson's proof (see \cite[Proposition 9]{Adam11}) and make certain parts more formal (in order to be able to formalize the proof of this proposition).
\begin{proposition}\label{prop:boxkRk}
	For each $k \in \nat$, $\sR^{(k)}_T(x) \in \cF^k_{T}(\EA)$.
\end{proposition}
\begin{proof}
The case $k = 0$ is trivial, since $\sR^{(0)}_T(x) \in \cF(\EA) = \cF^0_{T}(\EA)$. By Lemma \ref{lm:Fkcompos} it is sufficient to prove $\sR_T(x) \in \cF_{T}(\EA)$. Here we rely on the fact that the definition of $\sR_T(x)$ mentions $\Delta_0$-formulas rather than $\Delta_0(\exp)$, since in this case the truth predicate is elementary.
Consider the following elementary formula 
$$
\psi(x, y) \leftrightharpoons
\forall \delta \leqslant x \left(
\left(\delta(\cdot) \in \Delta_0 \land \exists p \leqslant x\, 
\Prf_T(p, \gn{\exists z\, \delta(z)})\right) \to
\exists u \leqslant y\, \mathsf{Tr}_{\Delta_0}(\gn{\delta(\num{u})})
\right).
$$
The graph of $\sR_T(x)$ is then defined by 
$\psi(x, y) \land \forall z < y\, \neg \psi(x, z).$
We prove by induction on $x$ (which is formalizable in $\EA$, since we have an elementary bound on the g{\"o}delnumber of the corresponding $T$-proof) that
$$
\EA \vdash \forall x\, \Box_T \left(\exists y\, \psi(\num{x}, y)
\right),
$$
whence $\EA \vdash \forall x\, \Box_T \exists y \left(
\psi(\num{x}, y) \land \forall z < y\, \neg \psi(\num{x}, z)\right)$,
follows by the least element principle applied under $\Box_T$ (since $T$ contains $\EA$ and $\psi(x, y)$ is elementary).

We argue in $\EA$ by induction on $x$ as follows. Assuming $\forall z < x\, \Box_T \left( \exists y\, \psi(\num{z}, y)\right)$ we derive 
$\Box_T \left(\exists y\, \psi(\num{x}, y)\right)$. We argue informally under $\Box_T$. By the induction hypothesis we have some $v$ such that 
$$
\forall \delta < \num{x} \left(
\left(\delta(\cdot) \in \Delta_0 \land \exists p \leqslant \num{x}\, 
\Prf_T(p, \gn{\exists z\, \delta(z)})\right) \to
\exists u \leqslant v\, \mathsf{Tr}_{\Delta_0}(\gn{\delta(\num{u})})
\right).
$$
Now, assuming $\left(\delta(\cdot) \in \Delta_0 \land \exists p \leqslant \num{x}\, 
\Prf_T(p, \gn{\exists z\, \delta(z)})\right)$,
we see that this $\delta$ is standard (since its proof $p \leqslant \num{x}$ is), whence 
by the small reflection principle we get
$\exists t\, \mathsf{Tr}_{\Delta_0}(\gn{\delta(\num{y})})$. 
Let $y$ be the maximum of $v$ and such $t$. It follows that $\psi(\num{x}, y)$, whence $\exists y\, \psi(\num{x}, y)$, as required
\end{proof}

The next proposition shows that $\sR_T(x) \notin \cF(T)$. Note that the statement and the proof is slightly more cumbersome than the one presented in \cite[Lemma 4]{Adam11}, since we define $\sR_T(x)$ in terms of the g{\"o}delnumbers of proofs rather than the number of symbols.

\begin{proposition}\label{prop:Rdomprovtot}
There is a monotone elementary function $p(x)$ such that
for each function $f(x) \in \cF(T)$ the function $\max_{y \leqslant x} f(y)$ is dominated by $\sR_T(p(x))$. Consequently, each monotone function $f(x) \in \cF(T)$ is dominated by $\sR_T(x)$.
\end{proposition}
\begin{proof}
Fix an arbitrary $f(x) \in \cF(T)$ and a constant $c \in \nat$ such that $T \vdash_c \forall x\, \exists y\,  \psi(x, y)$, where $\psi(x, y)$ is a $\Sigma_1$-definition of the graph of $f(x)$. Define $\hat{f}(x)$ to be $\max_{y \leqslant x} f(y)$. Note that $\hat{f}(x)$ is not necessarily provably total in $T$, however, for each $n \in \nat$ there is a short (relative to $n$) $T$-proof of  a $\Sigma_1$-sentence $\theta_n$, which naturally expresses $\exists y\, (\hat{f}(n) = y)$, e.g.,
$$
\theta_n \leftrightharpoons \exists z\, \exists y_0 \leqslant z\, \dots \exists y_n \leqslant z \left(\psi(\num{0}, y_0) \land \dots \land \psi(\num{n}, y_n) \right).
$$
Note that the least witness of $\theta_n$ is at least $\hat{f}(n)$.
Given $n \in \nat$ and the above proof of the totality of $f(x)$ one can elementarily obtain a $T$-proof of $\theta_n$. Let $q(c, n)$ be a strictly monotone elementary function $q(c, n)$ which, given an upper bound $c$ on the $T$-proof of the totality of some function $f(x)$, outputs an upper bound on the $T$-proof of a $\Sigma_1$-sentence $\theta_n$. Finally, define $p(x)$ to be $q(x, x)$. By the definition of $\sR_T(x)$ and monotonicity of $q(c, n)$ it follows that
$$
f(n) \leqslant \hat{f}(n) \leqslant \sR_T(q(c, n)) \leqslant \sR_T(p(n))
$$
holds for all $n \geqslant c$, as required.

Now, assuming $f(x) \in \cF(T)$ is monotone, apply the argument above to the monotone function $f(p(x+1)) \in \cF(T)$ (since $p(x)$ is elementary). Let $p^{-1}(x)$ to be $\max \{y \mid p(y) \leqslant x\}$ or $0$, if there is no such $y$. 
Using the definition of $p^{-1}(x)$ and monotonicity of $f(x)$, $p(x)$ and $\sR_T(x)$ we obtain
$$
f(n) \leqslant f(p(p^{-1}(n) + 1)) \leqslant \sR_T(p(p^{-1}(n))) \leqslant \sR_T(n),
$$
for all sufficiently large $n \in \nat$.

\end{proof}

The following lemma asserts the well-known fact about the class of provably total computable functions. We show that it also holds for all the classes $\cF^k_{T}(U)$.

\begin{lemma}\label{lm:Fkuniq}
For each $k \in \nat$ and each $F(x) \in \cF^k_{T}(U)$ there is a $\Sigma_1$-formula $\psi(x, y)$ defining the graph of $F(x)$ in the standard model such that
$$
U \vdash \forall x\, \Box^{(k)}_T \left(\exists ! y\, \psi(\num{x}, y)\right).
$$
\end{lemma}
\begin{proof}
Assume $F(x) \in \cF^k_{T}(U)$ and fix a corresponding $\Sigma_1$-formula $\exists z\, \delta(x, y, z)$ where $\delta(x, y, z)$ is a $\Delta_0$-formula. Define $\psi(x, y)$ to be 
$$
\exists u\, \exists z \leqslant u \left(u = \langle y, z \rangle \land \delta(x, y, z) \land \forall v < u\, \forall y', z' \leqslant v \left( v = \langle y', z' \rangle \to \neg \delta(x, y, u)\right)\right).
$$ 
This formula still defines the graph of $F(x)$ in the standard model.
Note that by the definition of $\psi(x, y)$ we have
$$
\EA \vdash \forall x\, \forall y_1\, \forall y_2 \left(\psi(x, y_1) \land \psi(x, y_2) \to y_1 = y_2\right),
$$
since there is only one least (w.r.t. its code) pair satisfying $\delta(x, y, z)$ and hence only one first component $y$ of such a pair.
We have
$U \vdash \forall x\, \Box^{(k)}_T \exists y\, \exists z\, \delta(\num{x}, y, z)$, so
$$
U \vdash \forall x\, \Box^{(k)}_T \exists u\, \exists y \leqslant u\, \exists z \leqslant u \left( u = \langle y, z \rangle \land \delta(\num{x}, y, z) \right), 
$$
whence by the least element principle (under $\Box_T$ for $k > 0$ and directly in $U$ for $k = 0$, since both $T$ and $U$ contain $\EA$) we get
$U \vdash \forall x\, \Box^{(k)}_T \exists y\, \psi(x, y)$, as required.
	
\end{proof}

The following four lemmas and the corollaries are used in the proof of the main theorem (Theorem \ref{th:main}) to obtain the desired improvement of an upper bound in Adamsson's condition.

\begin{lemma}\label{lm:cumuldom}
For each $k \in \nat$, every function $G(x) \in \cF^k_{T}(U)$ is dominated by a cumulative function $F(x) \in \cF^k_{T}(U)$.
\end{lemma}
\begin{proof}
Assume $G(x) \in \cF^k_{T}(U)$ and fix a $\Sigma_1$-formula $\psi(x, y)$ defining the graph of $G(x)$ such that 
$$
U \vdash \forall x\, \Box^{(k)}_T \left(\exists y\, \psi(\num{x}, y) \right).
$$
%if $k > 0$ and $U \vdash \forall x\, \exists y\, \psi(x, y)$, if $k = 0$.
By provable $\Sigma_1$-completeness it follows that 
$U  \vdash \forall x\, \Box^{(k)}_T \left(\exists y\, \Box_T \psi(\num{x}, \num{y}) \right),
$ 
whence 
$$
U  \vdash \forall x\, \Box^{(k)}_T \left(\exists p\, \exists y \leqslant p\,  \Prf_T(p, \gn{\psi(\num{x}, \num{y})}) \right).
$$

Let $F(x)$ be the function which maps $n$ to the g{\"o}delnumber $p$ of the least $T$-proof of $\psi(\num{n}, \num{m})$ for some $m \in \nat$. The above $U$-derivation ensures that $F(x) \in \cF^k_{T}(U)$ (apply the least element principle available both under the innermost $\Box_T$ (for $k \geqslant 1$) and in $U$ (for $k = 0$) to derive the existence of the least such $p$). For each $n \in \nat$, the value $F(n)$ is a g{\"o}delnumber of a $T$-proof of $\psi(\num{n}, \num{m})$ for some $m \in \nat$. Since $\psi(x, y)$ defines the graph of $G(x)$ and $T$ is $\Sigma_1$-sound, it must be the case that $m = G(n)$, whence 
$G(n) \leqslant F(n)$, since $m$ is certainly not greater than a g{\"o}delnumber of a proof containing $\num{m}$. The computation of the value $F(n)$ consists of looking through all $T$-proofs with g{\"o}delnumbers $p \leqslant F(n)$ and checking the elementary condition $\exists y \leqslant p\,  \Prf_T(p, \gn{\psi(\num{n}, \num{y})})$ for each such proof. This complexity of this process is elementary in $F(n)$, whence the cumulativity of $F(x)$ follows.
\end{proof}

\begin{lemma}\label{lm:proofbox}	
	There is an elementary function $u(x)$ such that for all sentences $\psi, \chi$ and $n \in \nat$,
	$$
	\left(S \vdash_n \psi \to \chi\right) \lor 
	\left(S \vdash_n \psi \to \neg \chi \right) \Longrightarrow
	S \vdash_{u(n)} \Box_S(\psi \to \chi) \lor \Box_S (\psi \to \neg \chi). 
	$$
\end{lemma}
\begin{proof}
	By provable $\Sigma_1$-completeness (and basic predicate logic) the following is provable in $\EA$,
	$$
	\forall p\, \forall \psi\, \forall \chi \left(\exists q \leqslant p\, 
	\Prf_S(q, \gn{\psi \to \chi}) \lor \Prf_S (q, \gn{\psi \to \neg \chi})
	\to \Box_S \left(\Box_S(\psi \to \chi) \lor \Box_S (\psi \to \neg \chi)\right) \right),
	$$
	whence the function $p(x, y, z)$ which, given g{\"o}delnumbers $y$ and $z$ of $\psi$ and $\chi$, and an upper bound $x$ on the g{\"o}delnumber  of an $S$-proof of one of implications $\psi \to \chi$ or $\psi \to \neg \chi$, outputs a g{\"o}delnumber of the least $S$-proof of the disjunction of boxes is elementary (being provably total in $\EA$). The required function $u(x)$ can be then defined as follows  
	$$
	u(x) = \sum_{y \leqslant x}\sum_{z \leqslant x} p(x, y, z), 
	$$     
	since the g{\"o}delnumbers of $\chi$ and $\psi$ are bounded by $n$, given
	$S \vdash_n \psi \to \chi$ or 
	$S \vdash_n \psi \to \neg \chi$.
\end{proof}

\begin{lemma}\label{lm:disjR}
There is a strictly monotone elementary function $v(x)$ such that 
for all $\psi, \chi$ and $n \in \nat$ we have
$$
S \vdash_n \Box_S \psi \lor \Box_S \chi \Longrightarrow
\left(S \vdash_{\sR_S(v(n))} \psi \right) \lor
\left(S \vdash_{\sR_S(v(n))} \chi \right).
$$
\end{lemma}
\begin{proof}
Let $v'(x)$ be the function which transforms the proof 
$\Box_S \psi \lor \Box_S \chi$ into the proof of 
$$
\exists p \left(\Prf_T(p, \gn{\psi}) \lor \Prf_T(p, \gn{\chi})\right),
$$ 
which is an elementary task. Define $v(x)$ to be some strictly monotone elementary upper bound on $v'(x)$, e.g., $v(x) = 2^x_k$ for some fixed $k \in \nat$. The required implication then follows by the definition of  $\sR_S(x)$ and the remark concerning the use of $\Delta_0$-formulas
rather than elementary formulas in it.
\end{proof}

\begin{lemma}\label{lm:fRltRc}
There is an elementary function $w(x, y)$ such that 
for each $f(x) \in \cF(T)$ there is a constant $c \in \nat$ such that
$$
f(\sR_S(n)) \leqslant \sR_S(w(c, n))
$$
holds for all $n \in \nat$. 
\end{lemma}
\begin{proof}
Fix a $\Sigma_1$-formula $\exists z\, \delta(x, y, z)$, which defines
the graph of $f(x)$ in the standard model, where $\delta(x, y, z)$ is 
a $\Delta_0$-formula and
$$
S \vdash \forall x\, \exists ! y\, \exists z\, \delta(x, y, z).
$$
Fix $n \in \nat$ and let $\sR_S(n) = m$. Fix a $\Sigma_1$-sentence $\exists x\, \psi(x)$, where $\psi(x)$ is a $\Delta_0$-formula, such that 
$S \vdash_n \exists x\, \psi(x)$ and the least witness for this sentence is $m$.  Since $S$ extends $\EA$, it follows that $S \vdash \exists x\, \left( \psi(x) \land \forall u < x\, \neg \psi(u) \right)$, and the g{\"o}delnumber of this proof is elementary in that of~$\exists x\, \psi(x)$.

Our aim is to construct a short (relative to $n$) $S$-proof of a $\Sigma_1$-sentence with the least witness greater than $f(m)$. 
Define a $\Sigma_1$-sentence $\sigma$ as follows
$$
\sigma \leftrightharpoons \exists v\, \exists x, y, z \leqslant v
\left(\psi(x) \land \forall u < x\, \neg \psi(u) \land \delta(x, y, z)\right).
$$
Informally, $\sigma$ asserts that the value $f(m)$, where $m$ is the least witness for $\exists x\, \psi(x)$, is defined. It follows then that the least witness for $\sigma$ is at least $f(m)$.

The required $S$-proof of $\sigma$ is obtained as follows: concatenate the proof of $\forall x\, \exists y\, \exists z\, \delta(x, y, z)$ (the totality of $f(x)$) with the proof of $\exists x \left(\psi(x) \land \forall u < x\, \neg \psi(u) \right)$ and derive $\sigma$ using basic predicate logic. The g{\"o}delnumber $w'(c, n)$ of this proof is elementary in 
g{\"o}delnumbers of the proofs of the totality of $f(x)$ and $\sigma$. As usual let $w(x, y) = \sum_{u \leqslant x}\sum_{v \leqslant y} w'(u, v)$.
It follows that $S \vdash_{w(c, n)} \sigma$ and 
$f(\sR_S(n)) \leqslant \sR_S(w(c, n))$ for all $n \in \nat$. 
\end{proof}

\begin{corollary}\label{cor:fRltRc}
There is a strictly monotone elementary function $h(x)$ such that for each
$f(x) \in \cF(T)$ the inequality
$$
f(\sR_S(n)) \leqslant \sR_S(h(n))
$$
holds for all sufficiently large $n \in \nat$. 
More generally, for each $k \in \nat$ and each $f(x) \in \cF(T)$ we have
$$
f(\sR^{(k+1)}_S(n)) \leqslant \sR^{(k+1)}_S(h(n))
$$
for all sufficiently large $n \in \nat$. 
\end{corollary}
\begin{proof}
Let	$w(x, y)$ be an elementary function constructed in Lemma 
\ref{lm:fRltRc}. Let $h(x) = 2^{2x}_m$, where $m \in \nat$ is such that
$w(x, y) \leqslant 2^{x + y}_m$ for all $x, y \in \nat$. Using monotonicity of $\sR_S(x)$ we get 
$$
f(\sR_S(n)) \leqslant \sR_S(w(c, n)) \leqslant \sR_S(2^{c + n}_m) 
\leqslant \sR_S(h(n))
$$
for all $n \geqslant c$, as required. The general case is then proved by induction on $k$, where the base case $k = 0$ has been just proved above.

Assume that for each $f(x) \in \cF(T)$ there is some $c \in \nat$ such that $f(\sR^{(k+1)}_S(n)) \leqslant \sR^{(k+1)}_S(h(n))$ holds
for all  $n \geqslant c$. Fix an arbitrary function $g(x) \in \cF(T)$. We aim at proving that
$$
g(\sR^{(k+2)}_S(n)) \leqslant \sR^{(k+2)}_S(h(n))
$$
holds for all sufficiently large $n$. Firstly, by the base case, we have $g(\sR_S(n)) \leqslant \sR_S(h(n))$ for all sufficiently large $n$, whence
$$
g(\sR^{(k+2)}_S(n)) = g(\sR_S(\sR^{(k+1)}_S(n))) \leqslant
\sR_S(h(\sR^{k+1}_S(n))),
$$
for all sufficiently large $n$, since is $\sR^{(k+1)}_S(x)$ monotone 
and unbounded. Now, applying the induction hypothesis to $h(x) \in \cF(T)$ (since $h(x)$ is elementary and $T$ extends $\EA$) we get
$$
h(\sR^{(k+1)}_S(n)) \leqslant \sR^{(k+1)}_S(h(n))
$$
for all sufficiently large $n$, whence, by monotonicity of $\sR_S(x)$, we obtain
$$
g(\sR^{(k+2)}_S(n)) \leqslant \sR_S(h(\sR^{(k+1)}_S(n))) \leqslant \sR^{(k+2)}_S(h(n))
$$
for all sufficiently large $n$, as required. 
\end{proof}

Below we prove several properties of the classes $\cF_{T}(U)$ (we consider only the case $k = 1$) similar to that of the class of provably total computable functions.
\begin{lemma}\label{lm:pi1external}
	If $X$ is any set of true $\Pi_1$-sentences, then
	$
	\cF_{T}(U) = \cF_{T}(U + X).
	$
\end{lemma}
\begin{proof}
	Let $\psi_f(x, y)$ be a $\Sigma_1$-formula defining the graph of $f(x)$ such that
	$$
	U + X \vdash \forall x\, \Box_T \left(\exists y\, \psi_f(\num{x}, y)\right).
	$$
	Let $\pi$ be a conjunction of finitely many $\Pi_1$-sentences from $X$ used in this proof. We then have
	\begin{align*}
	U  \vdash \pi & \to \forall x\, \Box_T \left(\exists y\, \psi_f(\num{x}, y)\right)\\
	& \to  \forall x\, \Box_T \left(\pi \to \exists y\, \psi_f(\num{x}, y)\right).
	\end{align*}
	But since $\neg \pi$ is (provably equivalent to) a $\Sigma_1$-sentence, by provable $\Sigma_1$-completeness we obtain
	\begin{align*}
	U  \vdash \neg \pi & \to \Box_T \neg \pi\\
	& \to  \forall x\, \Box_T \left(\pi \to \exists y\, \psi_f(\num{x}, y)\right).
	\end{align*}
	Thus, we have $U \vdash \forall x\, \Box_T\exists y \left(\pi \to \psi_f(\num{x}, y)\right)$, whence we obtain $f(x) \in \cF_{T}(U)$, since the formula $\pi \to \psi_f(x, y)$ still defines the graph of $f(x)$ 
	and is provably equivalent to $\Sigma_1$-formula, because $\pi$ is a true $\Pi_1$-sentence.
\end{proof}

We say that $S$ is a \emph{$\Pi_1$-extension of} $T$ if
$
\mathsf{Ax}_U(x) \leftrightharpoons \mathsf{Ax}_T(x) \lor \delta(x),
$
where $\delta(x)$ is an elementary formula such that $\nat \models 
\delta(\num{n})$ implies $n = \gn{\pi}$ for some $\pi \in \Pi_1$.

\begin{lemma}\label{lm:pi1extF}
	If $S$ is a $\Pi_1$-extension of $T$, then
	$
	\cF_{S}(U) = \cF_{T}(U).
	$
\end{lemma}
\begin{proof}
	Let $\psi_f(x, y)$ be a $\Sigma_1$-formula defining the graph of $f(x)$ such that
	$$
	U \vdash \forall x\, \Box_S \left(\exists y\, \psi_f(\num{x}, y)\right).
	$$
	Let $\chi$ be a true (since $S$ is a $\Pi_1$-extension of $T$) $\Pi_1$-sentence $\forall \pi \left(\delta(\gn{\pi}) \to \pi \in \Pi_1\right)$.
	We show that $f(x) \in \cF_{T}(U + \chi)$, whence the result follows by Lemma \ref{lm:pi1external}.
	Using the formalized deduction theorem, provable $\Sigma_1$-completeness, $\chi$, and the fact that finite conjunction of $\Pi_1$-sentences is provably equivalent to a $\Pi_1$-sentence together with $\EA \vdash \forall \pi \in \Pi_1\, \Box_T \left(\pi \eqv \mathsf{Tr}_{\Pi_1}(\pi)\right)$, we derive
	\begin{align*}
	U + \chi &\vdash 
	\forall x\, \exists \pi\, \exists z \left[ 
	\forall i < z\, \delta(\gn{\pi_i})
	\land 
	\Box_T \left(
	\bigwedge_{i < z} \pi_i \to \exists y\, \psi_f(\num{x}, y)
	\right)\right]\\
	&\vdash 
	\forall x\, \exists \pi\, \exists z \left[\forall i < z\, \pi_i \in \Pi_1 \land \Box_T \left[ 
	\forall i < \num{z}\, \delta(\gn{\pi_i})
	\land 
	\left( 
	\bigwedge_{i < z} \pi_i \to \exists y\, \psi_f(\num{x}, y)
	\right)\right] \right]\\
	&\vdash 
	\forall x\, \exists \pi\, \exists z\, \Box_T \Bigl[ 
	\forall i < \num{z}\, \delta(\gn{\pi_i})
	\land 
	\Bigl( 
	\mathsf{Tr}_{\Pi_1}(\gn{\bigwedge_{i < z} \pi_i}) \to \exists y\, \psi_f(\num{x}, y)
	\Bigr)\Bigr]\\
	&\vdash 
	\forall x\,  \Box_T \exists \pi\, \exists z\, \Bigl[ 
	\forall i < z\, \delta(\gn{\pi_i})
	\land 
	\Bigl( 
	\mathsf{Tr}_{\Pi_1}(\gn{\bigwedge_{i < z} \pi_i}) \to \exists y\, \psi_f(\num{x}, y)
	\Bigr)\Bigr]\\
	&\vdash \forall x\, \Box_T \left(\exists y\, 
	\theta(\num{x}, y)
	\right),
	\end{align*}
	where $\theta(x, y) \leftrightharpoons
	\exists \pi\, \exists z\, \Bigl[ 
	\forall i < z\, \delta(\gn{\pi_i})
	\land 
	\Bigl( 
	\mathsf{Tr}_{\Pi_1}(\gn{\bigwedge_{i < z} \pi_i}) \to
	\psi_f(x, y)
	\Bigr)\Bigr]$ 
	is (provably equivalent to) a $\Sigma_1$-formula and still defines the graph of $f(x)$, since $S$ is an extension of $T$ with true $\Pi_1$-sentences.
\end{proof}

Lemma \ref{lm:pi1external} also allows to obtain the following alternative definition of the class $\cF_{T}(U)$.
\begin{proposition}
$f(x_1, \dots, x_m) \in \cF_{T}(U)$ if and only if there is a $\Sigma_1$-formula $\psi_f(x_1, \dots, x_m, y)$ defining the graph of $f$ such that  
the function which maps $n_1, \dots, n_m$ to the least g{\"o}delnumber of a $T$-proof
of $\exists y\, \psi_f(\num{n_1}, \dots, \num{n_m}, y)$
is provably total in $U$.
\end{proposition}
\begin{proof}
Only if part follows from the definition of $\cF_{T}(U)$. Fix a function $f(x)$ (we consider unary function for simplicity) and its $\Sigma_1$-definition $\psi_f(x, y)$. Let $\theta(x, p)$ be a $\Sigma_1$-definition of a function which maps $n$ to the least g{\"o}delnumber of a $T$-proof
of $\exists y\, \psi_f(\num{n}, y)$, such that
$U \vdash \forall x\, \exists p\, \theta(x, p)$.
Consider the following true $\Pi_1$-sentence 
$$
\pi \leftrightharpoons \forall x\, \forall p \left(
\theta(x, p) \to \Prf_T(p, \gn{\exists y\, \psi_f(\num{x}, y)}) \right).
$$
It follows that $U + \pi \vdash \forall x\, \exists p\, 
 \Prf_T(p, \gn{\exists y\, \psi_f(\num{x}, y)})$, whence
 $f(x) \in \cF_{T}(U + \pi) = \cF_{T}(U)$ by Lemma \ref{lm:pi1external}, as required.
\end{proof}

The following proposition gives a more explicit description of the class $\cF_{T}(U)$. We denote by $\cC_1 \circ \cC_2$ the class of all compositions of the form $f(g_1(\vec{x}), \dots, g_n(\vec{x}))$, where $f \in \cC_1$ and $g_1, \dots, g_n \in \cC_2$. It is always assumed that both $\cC_1$ and $\cC_2$ contain all projections (if it is not so, we add all the projections to these classes without writing it explicitly), 
so that both $\cC_1$ and $\cC_2$ are included in $\cC_1 \circ \cC_2$.
\begin{proposition}\label{prop:Fboxdef}
	$\cF_{T}(U) = \cF(T) \circ \{\sR_T\} \circ \cF(U) = \mathcal{E}_3 \circ \{\sR_T\} \circ \cF(U).$	
\end{proposition}
\begin{proof}
	Let us first show that $\cF(T) \circ \{\sR_T\} \circ \cF(U) \subseteq \cF_{T}(U)$. For simplicity we consider the following composition 
	$f(\sR_T(g_1(x)), \sR_T(g_2(x))) \in \cF(T) \circ \{\sR_T\} \circ \cF(U).$
	The general case (more functions and variables) can be dealt with similarly. Fix $\Sigma_1$-formulas $\chi_{g_1}(x, y)$, $\chi_{g_2}(x, y)$,
	$\theta_{\sR_T}(x, y)$ (this formula is given by Proposition \ref{prop:boxkRk}), $\phi_f(x, y, z)$ defining the graphs of the corresponding functions such that
	$$
	U \vdash \forall x\, \exists z_1\, \chi_{g_1}(x, z_1) \land 
	\forall x\, \exists z_2\, \chi_{g_2}(x, z_2) \land \forall z\, \Box_T \left(\exists u\, \theta_{\sR_T}(\num{z}, u)\right),
	$$
	and $T \vdash \forall u_1, u_2\, \exists y\, \phi_f(u_1, u_2, y)$.
	Define the graph of the above composition as follows
	$$
	\psi(x, y) \leftrightharpoons 
	\exists z_1, z_2, u_1, u_2
	\left( \chi_{g_1}(x, z_1) \land \chi_{g_2}(x, z_2) \land 
	\theta_{\sR_T}(z_1, u_1) \land \theta_{\sR_T}(z_2, u_2)
	\land \phi_f(u_1, u_2, y) \right).
	$$
	We claim that $U \vdash \forall x\, \Box_T \left(
	\exists y\, \psi(\num{x}, y)\right)$. Indeed, by provable $\Sigma_1$-completeness we have
	\begin{align*}
	U \vdash \chi_{g_1}(x, z_1) \land \chi_{g_2}(x, z_2) 
	&\to \Box_T \left(\chi_{g_1}(\num{x}, \num{z_1}) \land \chi_{g_2}(\num{x}, \num{z_2})  
	\land \exists u_1\, \theta_{\sR_T}(\num{z_1}, u_1) \land \exists u_2\, \theta_{\sR_T}(\num{z_2}, u_2)\right)\\
	&\to \Box_T \left(\exists y\, \psi(\num{x}, y)\right),
	\end{align*}
	where in the second implication we weakened the formula under $\Box_T$ by introducing the existential quantifiers instead of the numerals $\num{z_1}$ and $\num{z_2}$, and applied the $T$-provable totality of to get $\exists y\, \phi_f(u_1, u_2, y)$.
	So we have
	$$
	U \vdash \forall x \left(
	\exists z_1\, \chi_{g_1}(x, z_1) \land \exists z_2\, \chi_{g_2}(x, z_2) 
	\to \Box_T \left( \exists y\, \psi(\num{x}, y)\right)
	\right), 
	$$
	whence
	$
	U \vdash \forall x\, \Box_T\left( \exists y\, \psi(\num{x}, y)\right)
	$ follows, as required.
	
	The inclusion $\mathcal{E}_3 \circ \{\sR_T\} \circ \cF(U) \subseteq \cF(T) \circ \{\sR_T\} \circ \cF(U)$ is clear, since $T$ contains $\EA$.
	Now we prove the final inclusion  $\cF_{T}(U) \subseteq \mathcal{E}_3 \circ \{\sR_T\} \circ \cF(U)$. Fix some $f(x) \in \cF_{T}(U)$ (where for simplicity we again consider the unary function) and a corresponding $\Sigma_1$-formula $\psi_f(x, y)$ such that 
	$
	U \vdash \forall x\, \Box_T \left( \exists y\, \psi_f(\num{x}, y)\right).
	$
	Fix an elementary formula $\delta(x, y, z)$ such that 
	$\psi_f(x, y) \leftrightharpoons \exists z\, \delta(x, y, z)$. 
	Define $h(x)$ to be a function which maps $n$ to the g{\"o}delnumber of the least $T$-proof of 
	$\exists u\, \exists y, z \leqslant u\, \delta(\num{n}, y, z)$. The function $h(x)$ is $U$-provably total, since 
	$$
	U \vdash \forall x\, \Box_T \left( \exists y\, \exists z\, \delta(\num{x}, y, z)\right),
	$$
	and the least element principle for elementary formulas is available in $U$, since $U$ contains $\EA$. By $\EA$-provable $\Sigma_1$-completeness we have
	$$
	\EA \vdash \forall x, y, z \left(\delta(x, y, z) \to \Box_T \delta(\num{x}, \num{y}, \num{z})\right),
	$$
	whence the function which maps $n, m, k$ to the g{\"o}delnumber of the least $T$-proof of $\delta(\num{n}, \num{m}, \num{k})$ is elementary (being provably total in $\EA$). Let $t(x, y, z)$ be an elementary upper bound for this function which is monotone in all arguments (e.g., we can take $t(x, y, z)$ to be $2^{x + y + z}_l$ for some fixed sufficiently large $l \in \nat$).
	
	By definition of $\sR_T(x)$ and $h(x)$ the value $\sR_T(h(n))$ is an upper bound on the least witness to 
	$\exists u\, \exists y, z \leqslant u\, \delta(\num{n}, y, z)$.
	We claim that the graph of $f(x)$ can be defined as follows
	$$
	\exists p \leqslant t(x, \sR_T(h(x)), \sR_T(h(x)))\, \exists z \leqslant p\, \Prf_T(p, \gn{\delta(\num{x}, \num{y}, \num{z})}).
	$$
	Indeed, if the relation holds for some $n, m \in \nat$, then $T$ proves $\delta(\num{n}, \num{m}, \num{k})$ for some $k \in \nat$, whence $\exists z\, \delta(\num{n}, \num{m}, z)$ is true (since $T$ is $\Sigma_1$-sound) and so $f(n) = m$. Conversely, assume $f(n) = m$ and choose the least $k \in \nat$ such that $\delta(\num{n}, \num{m}, \num{k})$ is true. By the discussion above the least witness to 
	$\exists u\, \exists y, z \leqslant u\, \delta(\num{n}, y, z)$ is bounded by $\sR_T(h(n))$, whence $m$ (since $m$ is the only witness for $y \leqslant u$) and $k$ are also bounded by $\sR_T(h(n))$. By definition of the function $t(x, y, z)$ there is a $T$-proof of $\delta(\num{n}, \num{m}, \num{k})$ with the g{\"o}delnumber bounded by $t(n, m, k)$, which is itself bounded by $t(n, \sR_T(h(n)), \sR_T(h(n)))$ since $t(x, y, z)$ is monotone, whence the relation holds.
	
	Now, consider the following elementary function
	$$
	g(x, u) = \mu y_{\leqslant u}\left( \exists p \leqslant u\, \exists z \leqslant p\, \Prf_T(p, \gn{\delta(\num{x}, \num{y}, \num{z})}) \right).
	$$
	We claim that $f(x) = g(x, t(x, \sR_T(h(x)), \sR_T(h(x))))$, which belongs to $\mathcal{E}_3 \circ \{\sR_T\} \circ \cF(U)$ as a composition of
	$g(x, t(x, y, z)) \in \mathcal{E}_3$ and $\sR_T(h(x)) \in 
	\{\sR_T\} \circ \cF(U)$. The claim follows from the equivalent definition of the graph of $f(x)$ given above. For a given $n \in \nat$ the value $t(n, \sR_T(h(n)), \sR_T(h(n)))$ gives an upper bound on the g{\"o}delnumber of a $T$-proof (and hence the numerals appearing in this proof) of $\delta(\num{n}, \num{m}, \num{k})$ for some $m$ and $k$. Since $\exists z\, \delta(x, y, z)$ defines the graph of $f(x)$ it must be the case that $f(n) = m$.
\end{proof}

One can check that the above proof also goes through for $k > 1$ (since Proposition \ref{prop:boxkRk} also holds for such $k$), so we obtain $\cF^k_{T}(U) = \cF(T) \circ \{\sR^{(k)}_T(x)\} \circ \cF(U)$ for each $k \geqslant 1$, and $\cF^0_{T}(U) = \cF(U)$.

\section{The non-isomorphism theorem and its corollaries}

Shavrukov's sufficient condition for $\fD_T \not \cong \fD_S$ (Theorem \ref{th:shavnoniso}) implies that the theories $T$ and $S$ have distinct classes of provably total computable functions, namely, $\sR_S(x) \in \cF(T) \setminus \cF(S)$. As a corollary of the improved version (Theorem \ref{th:adamnoniso}) of non-isomorphism theorem, Adamsson obtained an example demonstrating that the condition $\cF(T) \neq \cF(S)$ is not necessary for $\fD_T\not \cong \fD_S$. In fact, he showed that 
$(\fL_T, \Box_T) \not \cong (\fL_T, \Box^{(6)}_T)$, and so the isomorphism class of $\fD_T$ depends on the choice of the provability predicate for $T$. However, there are no other applications of the improved non-isomorphism condition in the work of Adamsson \cite{Adam11}, which provide new non-isomorphism examples (which do not already follow from Shavrukov's condition).

In this section we strengthen Adamsson's improved version of Shavrukov's non-isomorphism theorem. Our strengthening allows to answer positively the question of Adamsson concerning the non-isomorphism $(\fL_T, \Box_T) \not \cong (\fL_T, \Box_T\Box_T)$. We also obtain new examples of theories with non-isomorphic diagonalizable algebras. In particular, we give a more natural example of theories with non-isomorphic algebras and same classes of provably total computable functions, and show that the condition of provable $\Pi_1$-equivalence is not sufficient for the isomorphism (cf. Theorem \ref{th:shaviso}).

Firstly, let us discuss the main ingredients of Shavrukov's proof \cite{Shav93a} of Theorem \ref{th:shavnoniso}
%and improvements made by Adamsson 
to clarify the proof of the main theorem which follows. 
The first major idea is that of associating a model of computation $\delta^T_\gamma(x)$ with the diagonalizable algebra $\fD_T$.
The role of natural numbers is played by the elements
$$
\sharp^n_T \leftrightharpoons \Box^{(n+1)}_T \bot \land \neg \Box^{(n)}_T \bot.
$$ 
The main property of these sentences is the following
$$
\EA \vdash \forall x\, \forall y \left(\sharp^x_T \land \sharp^y_T \to x = y\right).
$$ 
Programs correspond to sentences $\gamma$ and computations correspond to $T$-proofs. Computation complexity is measured by the size of a proof. 
For a given sentence $\gamma$ define a $\{0,1\}$-valued partial recursive function $\delta^T_\gamma(x)$ 
and the corresponding complexity measure $\Delta^T_\gamma(x)$ as follows
\begin{align*}
\delta^T_\gamma(n) &= \begin{cases}
0, &\text{ if $T \vdash \sharp^n_T \to \gamma$,}\\
1, &\text{ if $T \vdash \sharp^n_T \to \neg \gamma$,}\\
\uparrow, &\text{ otherwise,} 
\end{cases}\\
\Delta^T_\gamma(n) &= \mu p.\left(T \vdash_p \sharp^n_T \to \gamma  \lor T \vdash_p \sharp^n_T \to \neg \gamma \right).
\end{align*}

The main fact about the function $\delta^T_\gamma(x)$ is the following lemma (see \cite[Lemma 2]{Shav93a}), which asserts
that $\delta^T$ provides an indexing of $\{0, 1\}$-valued partial recursive functions, and there is a translation between $\delta^T$-indexes and $\phi$-indexes such that the corresponding complexity measures $\Delta^T$ and $\Phi$ are elementarily related.

\begin{lemma}\label{lm:index}
	For each $\phi$-index $c$ for a $\{0,1\}$-valued partial recursive function there is a $\delta^T$-index $\gamma$ such that
	$$
	\delta^T_\gamma = \phi_c \quad \mbox{and} \quad \Delta^T_\gamma \preccurlyeq_{\dom\, \phi_c} \Phi_c.
	$$
	Conversely, for each $\delta^T$-index $\gamma$ there is a $\phi$-index $c$ for a $\{0,1\}$-valued partial recursive function such that
	$$
	\phi_c = \delta^T_\gamma \quad \mbox{and} \quad \Phi_c \preccurlyeq_{\dom\, \delta^T_\gamma} \Delta^T_\gamma.
	$$
\end{lemma}

An important observation concerning $\delta$-indexes is that if $(\cdot)^* \colon \fD_T \to \fD_S$ is a homomorphism, then $\delta^S_{\gamma^*}$ is an extension of $\delta^T_\gamma$. If, moreover, $(\cdot)^*$ is an isomorphism, then 
$\delta^S_{\gamma^*} = \delta^T_\gamma$. So a homomorphism provides certain kind of translation between $\delta^T$-indexes and $\delta^S$-indexes. Note that, unlike in Lemma \ref{lm:index}, in this case there is a loss of the information about the size of proofs and, in particular, about the relationship between $\Delta^S_{\gamma^*}$ and $\Delta^T_\gamma$.

This is where the second crucial idea of Shavrukov comes into play. Namely, to apply a version of Blum's compression theorem for the model $\delta^T_\gamma(x)$ and the corresponding complexity measure $\Delta^T_\gamma(x)$.
The following lemma  is a version of Blum's compression theorem for $\phi$ and $\Phi$, which can be translated to $\delta$ and $\Delta$ by means of Lemma \ref{lm:index}. For the proof see \cite[Lemma 7]{Shav93a} (there it is stated in a weaker form, however, the proof yields a stronger result).

\begin{lemma}\label{lm:BlumPhi}
	For each cumulative function $f$ with $\dom\, f = X$ there is 
	a $\phi$-index $b$ for a $\{0,1\}$-valued partial recursive function 
	with $\dom\, \phi_b = X$ such that 
	$$
	\Phi_b \approx_X f \preccurlyeq_{X \cap \dom\, \phi_a} \Phi_a,
	$$ 
	whenever $\phi_a =_{X \cap \dom\, \phi_a} \phi_b$.
\end{lemma}

The next lemma is a special case of Lemma \ref{lm:BlumPhi} 
(see \cite[Lemma 3]{Shav93a}).
\begin{lemma}\label{lm:BlumX}
	For each r.e.\,set $X$ there is a $\phi$-index $c$ for a $\{0,1\}$-valued partial recursive function such that $\dom\, \phi_c = X$ and $\Phi_c \preccurlyeq_{X \cap \dom\, \phi_d} \Phi_d$, whenever $\phi_d =_{X \cap \dom\, \phi_d} \phi_c$.
\end{lemma}

The following lemmas give several basic results about the notion of cumulativity (see \cite[Lemma 5]{Shav93a} and \cite[Lemma 6]{Shav93a}).

\begin{lemma}\label{lm:deltacumul}
	For each sentence $\gamma$ the function $\Delta^T_\gamma$ is cumulative.
\end{lemma}

\begin{lemma}\label{lm:composcumul}
	If the functions $F$ and $G$ are cumulative, then so is $F \circ G$.	
\end{lemma}

We now have everything we need to prove the following main theorem, which is a mentioned strengthening of Adamsson's result.

\begin{theorem}\label{th:main}
	Let $T$ and $S$ be theories with standard provability predicates $\Box_T$ and $\Box_S$.
	If there is an epimorphism from $\fD_T$ onto $\fD_S$ then
	for each $k \in \nat$ there is a function $G(x) \in \cF^{k+1}_{S}(\EA)$ which is not dominated by any function 
	$F(x) \in \cF^k_{T}(T)$.
	More precisely, there is an elementary function $q(x)$ such that for each $k \in \nat$, for each function $F(x) \in \cF^k_{T}(T)$ the following inequality
	$$
	F(n) \leqslant \sR^{(k+1)}_S(q(n))
	$$
	holds for infinitely many $n \in \nat$.
\end{theorem}
\begin{proof}
	We follow Adamsson's proof (which itself follows Shavrukov's proof) very closely and make certain changes to obtain sharper bounds. The general line of the proof is unchanged, so we mainly concentrate on the parts where the modifications are made.

	Assume there is an epimorphism 
	$(\cdot )^* \colon \fD_T \to \fD_S$. As usual we consider $(\cdot)^*$ as a mapping on the set of sentences up to provable equivalence and write just $A^*$ for an arbitrary representative of the equivalence class $([A]_{\sim_T})^*$. 
	By Proposition \ref{lm:cumuldom} it is sufficient to get the conclusion of the theorem for cumulative functions. So fix an arbitrary cumulative function $F(x) \in \cF^k_{T}(T)$. 
	The proof consists of two main parts, which are embodied in the following two inequalities
	$$
	F \circ \Delta^S_\alpha \preccurlyeq_X \Delta^S_\beta \preccurlyeq_Y \sR^{(k+1)}_S \circ q \circ \Delta^S_\alpha,
	$$
	for some infinite sets $Y \subseteq X \subseteq \nat$, whence the result follows.

	Let $X$ be some r.e.\,nonrecursive set. Fix a  $\phi$-index $c$ given by Lemma \ref{lm:BlumX} for a partial recursive $\{0,1\}$-valued function such that $X = \dom\, \phi_c$ and $\Phi_c \preccurlyeq_{X \cap \dom\, \phi_d} \Phi_d$, whenever $\phi_d =_{X \cap \dom\, \phi_d} \phi_c$. Using Lemma \ref{lm:index}
	find a $\delta^S$-index $\alpha$ such that $\phi_c = \delta^S_\alpha$ and $\Delta^S_\alpha \preccurlyeq_X \Phi_c$. Since $(\cdot)^*$ is surjective, there is some sentence $A$ with $A^* = \alpha$. Fix such a sentence $A$. 
	This auxiliary set $X$ and the corresponding sentences $\alpha$ and $A$ are needed in order to preserve some part of the information concerning the size of proofs, which otherwise gets lost when we move from $T$ into $S$ using the homomorphism (which is one of the most significant places in the following proof). 
	Since $(\cdot)^*$ is a homomorphism, the function $\delta^S_\alpha$ is an extension of $\delta^T_A$. By the choice of $\alpha$ we have $\Delta^S_\alpha \preccurlyeq_X \Phi_c  \preccurlyeq_{\dom\, \delta^T_A}
	\Phi_d \preccurlyeq_{\dom\, \delta^T_A} \Delta^T_A$,
	where $d$ is some $\phi$-index for $\delta^T_A$ given by Lemma~\ref{lm:index}. Fix some strictly monotone elementary function $p(x)$ such that $\Delta^S_\alpha \leqslant_{\dom\, \delta^T_A} p \circ \Delta^T_A$.
	
	In order to be able to formalize certain facts about the function $\Delta^S_\alpha(x)$ inside $T$
	we consider the following modified program ($\Sigma_1$-formula defining the graph) for this function. 
	Define $\psi(x, y)$ as $\psi'(x, y) \land \forall z < y\, \neg \psi'(x, z)$, where 
	\begin{align*}
	\psi'(x, y) \leftrightharpoons 
	\Prf_S(y, \gn{\sharp^x_S \to \alpha}) \lor 
	\Prf_S(y, \gn{\sharp^x_S \to \neg \alpha})\ &\lor\\
	\lor\ \exists z \leqslant y\, \left( 
	p(z) = y \lor 
	\Prf_T(y, \gn{\sharp^x_T \to A})\right. &\lor \left.
	\Prf_T(y, \gn{\sharp^x_T \to \neg A})
	\right).
	\end{align*}
	Informally, $\psi(x, y)$ asserts that $y$ is the least number such that
	$\Delta^S_\alpha(x) = y$ or $y = p(z)$ for some $z$ such that $\Delta^T_A(x) = z$. Note that $\psi(x, y)$ is an elementary formula.
	
	The inequality $\Delta^S_\alpha \leqslant_{\dom\, \delta^T_A} p \circ \Delta^T_A$ implies that the formula $\psi(x, y)$ indeed defines the graph of $\Delta^S_\alpha(x)$ in the standard model.
	Note that (the least element principle for elementary formulas available in $T$ is used here)
	$$
	T \vdash \forall x \left( \exists z \left(p(\Delta^T_A(x)) = z \right) \to \exists y\, \psi(x, y)\right),
	$$
	here $p(\Delta^T_A(x)) = z$ is an elementary formula defining the graph of $p(\Delta^T_A(x))$ in the most natural way.
	By the definition of $\psi(x, y)$ we also have
	$$
	\EA \vdash \forall x\, \forall y_1\, \forall y_2 \left(
	\psi(x, y_1) \land \psi(x, y_2) \to y_1 = y_2
	\right).
	$$
	
	Consider the function $F \circ \Delta^S_\alpha$, which is cumulative by Lemma \ref{lm:composcumul}, Lemma \ref{lm:deltacumul} and since $F(x)$ is assumed to be cumulative. By Lemma \ref{lm:BlumPhi} there is a $\phi$-index $b$ for a partial recursive $\{0,1\}$-valued function such that
	$F \circ \Delta^S_\alpha \approx_X \Phi_b$ and $\Phi_b \preccurlyeq_{X \cap \dom\, \phi_a} \Phi_a$, given $\phi_a =_{X \cap \dom\, \phi_a} \phi_b$.
	
	By Lemma \ref{lm:index} there is some $\delta^S$-index $\beta$ with $\phi_b = \delta^S_\beta$, whence $F \circ \Delta^S_\alpha \preccurlyeq_X \Delta^S_\beta$ which gives us the first of the two main inequalities, however, an arbitrary chosen sentence $\beta$ with this property bears no relationship to the sentence $\alpha$, and we will not be able to obtain the second inequality, which requires that $\Delta^S_\beta$ grows not too much faster than $\Delta^S_\alpha$. That's why we need a specifically constructed sentence with this property in order to be able to formalize in $T$ certain results, which together with the homomorphism yield the desired relationship between $\Delta^S_\alpha$ and~$\Delta^S_\beta$.

	We define a specific sentence $B$ such that $\Delta^T_B$ is $T$-provably related to $\Delta^T_A$ (the precise meaning of this is given below), whence $\Delta^T_\beta$ will be $S$-provably related to $\Delta^S_\alpha$, where $\beta = B^*$. This will allow us to obtain an upper bound on $\Delta^S_\beta$ in terms of $\Delta^S_\alpha$, as required.

	Fix an elementary function $s(x)$ such that 
	$\Phi_b \leqslant_X s \circ F \circ \Delta^S_\alpha,$
	and define the formula
	$$
	B(x) \leftrightharpoons \exists u \left(s(F(\Delta^S_\alpha(x))) = u \land \Phi_b(x) \leqslant u \right) \to \phi_b(x) = 0,
	$$
	where we express the relation $s(F(\Delta^S_\alpha(x))) = u$
	by the following $\Sigma_1$-formula
	$$
	\exists y\, \exists z \left(
	\psi(x, y) \land \chi(y, z) \land \theta(z, u)
	\right),
	$$
	where $\theta(x, y)$ is a $\Sigma_1$-formula defining the graph of $s(x)$ such that $\EA \vdash \forall x\, \exists ! y\, \theta(x, y)$, and
	$\chi(x, y)$ is a $\Sigma_1$-formula given by Lemma \ref{lm:Fkuniq} applied to the function $F(x)$. In particular, we have
	$T \vdash \forall x\, \Box^{(k)}_T(\exists y\, \chi(\num{x}, y))$, and  $ \EA \vdash \forall x\, \forall y_1\, \forall y_2 \left(\chi(x, y_1) \land \chi(x, y_2) \to y_1 = y_2\right)$.
	Since the witness for $\exists y\, \psi(x, y)$ is also $\EA$-provably unique by the definition of $\psi(x, y)$, it follows that
	$$
	\EA \vdash 
	\forall x\, \forall u_1\, \forall u_2 \left(
	\left(s(F(\Delta^S_\alpha(x))) = u_1 \land s(F(\Delta^S_\alpha(x))) = u_2
	\right) \to u_1 = u_2 \right).
	$$
	Note that we represent the graph of   $\Delta^S_\alpha(x)$ not in the natural way, but using a modified formula $\psi(x, y)$ constructed above. This will become important later when we formalize certain parts of the following argument in $T$.

	Define the sentence $B \leftrightharpoons \forall x \left(\sharp^x_T \to B(x)\right)$.
	We claim that $\delta^T_B =_X \phi_b$. Given $n \in X$, we have 
	$s(F(\Delta^S_\alpha(n))) = m$ and $\Phi_b(n) \leqslant m$, for some $m \in \nat$ (since $\dom\, \Delta^S_\alpha = X$ and by the choice of $s(x)$), whence the antecedent of $B(\num{n})$ is true and provable (being a $\Sigma_1$-sentence) and 
	$$
	T \vdash B(\num{n}) \eqv \phi_b(\num{n}) = 0.
	$$
	If $\phi_b(n) = 0$, then $T \vdash B(\num{n})$,
	since $T$ proves (a true $\Sigma_1$-sentence) $\phi_b(\num{n}) = 0$.
	In this case
	\begin{align*}
	T \vdash \sharp^n_T &\to \forall x \left( \sharp^x_T \to \left(x = \num{n} \land B(\num{n})\right) \right)\\
	&\to \forall x \left( \sharp^x_T \to B(x)\right)\\
	&\to B,
	\end{align*} 
	whence $\delta^T_B(n) = 0$. If $\phi_b(n) = 1$, then $T \vdash \neg B(\num{n})$, since $T$ proves (a true $\Sigma_1$-sentence) $\phi_b(\num{n}) = 1$ which implies $\neg (\phi_b(\num{n}) = 0)$.
	In this case
	\begin{align*}
	T \vdash \sharp^n_T &\to \left(\sharp^n_T \land \neg B(\num{n}) \right)\\
	&\to \exists x \left( \sharp^x_T \land \neg B(x)\right)\\
	&\to \neg B,
	\end{align*} 
	whence $\delta^T_B(n) = 1$. Since $(\cdot)^*$ is a homomorphism,
	for $\beta = B^*$ we have $\delta^T_B \subseteq \delta^S_\beta$,
	and so $\phi_b =_X \delta^T_B =_X \delta^S_\beta$. The choice of $b$ implies that 
	$$
	F \circ \Delta^S_\alpha \approx_X \Phi_b \preccurlyeq_X \Phi_a \preccurlyeq_X \Delta^S_\beta,
	$$
	where $a$ is a $\phi$-index of $\delta^S_\beta$ given by Lemma \ref{lm:index}, which gives the first of the two main inequalities
	$F \circ \Delta^S_\alpha \preccurlyeq_X \Delta^S_\beta$.
	
	In order to be able to obtain an upper bound on $\Delta^S_\beta$ in terms of $\Delta^S_\alpha$ we need to formalize in $T$ the relationship between $\Delta^T_A$ and $\Delta^T_B$ present in the proof of the claim $\delta^T_B =_X \phi_b$ given above. 
	We have the following derivation
	\begin{align*}
	T \vdash \forall x\, \Bigl(\Box_T \left(\sharp^x_T \to A \right) \lor \Box_T \left(\sharp^x_T \to \neg A \right) \Bigr. &\to \exists z\, (p(\Delta^T_A(x)) = z) 
	\\
	&\to \exists y\, \psi(x, y)\\
	%&\to \Delta^S_\alpha(x)\dar\\
	&\to \Box^{(k)}_T \exists u \left(s(F(\Delta^S_\alpha(\num{x}))) = u\right)\\
	&\to \Box^{(k)}_T \left(\Box_T B(\num{x}) \lor \Box_T \neg B(\num{x})\right)\\
	&\to \left. \Box^{(k)}_T \left(\Box_T\left(\sharp^x_T \to B\right) \lor \Box_T \left(\sharp^x_T \to \neg B\right) \right)\right).
	\end{align*}
	Let us make some comments on each of the implications above. The first one follows by the definition of $\Delta^T_A(x)$ (and we use the most natural formalization of it inside $T$), elementary induction (to get the existence of the least $T$-proof) and $T$-provable totality of $p(x)$ (since $p(x)$ is elementary). The second one was discussed right after the definition of $\psi(x, y)$ (this is the place where we need a modified definition of the graph of $\Delta^S_\alpha(x)$ and not the natural one we used for $\Delta^T_A(x)$ in order for the formalization to go through). As for the third one, we use 
	provable $\Sigma_1$-completeness (if $k \geqslant 1$), 
	the above property of $\chi_F(x, y)$
	and $T$-provable totality of (an elementary function) $s(x)$ to get 
	\begin{align*}
	T \vdash \forall x\, \forall y\, \Bigl(
	\psi(x, y) \Bigr. &\to \Box^{(k)}_T \left(
	\psi(\num{x}, \num{y}) \land \exists z \left(\chi_F(\num{y}, z) \right) 
	\right) \\
	&\to \left. \Box^{(k)}_T \exists u \left(s(F(\Delta^S_\alpha(\num{x}))) = u\right)\right).
	\end{align*}
	The proof of the final implication is unchanged (as compared to the argument above), since it uses only basic predicate logic and provable properties of the formulas $\sharp^x_T$.
	
	Now, for the fourth implication we need to be more careful, since we do not have the facts like $\Phi_B \preccurlyeq_X s \circ F \circ \Delta^S_\alpha$ and $\dom\, \phi_b = X$ inside $T$. In the following argument we rely on the fact that the ternary relation $\Phi_e(x) \leqslant y$ is expressed using an elementary formula.
	We argue informally under $\Box_T$ as follows. Assume $\exists u \left(s(F(\Delta^S_\alpha(\num{x}))) = u \right)$. Fix such a number $u$ and consider two cases: $\Phi_b(\num{x}) \leqslant u$ (in this case $\phi_b(\num{x})$ is defined) and $\Phi_b(\num{x}) > u$ (in this case it may be that $\phi_b(\num{x})$ is undefined). Assume $\Phi_b(\num{x}) \leqslant u$, we then have $\exists y \left(\phi_b(\num{x}) = y\right)$. Let us fix such a $y$ and consider two more cases: $y = 0$ or $y \neq 0$. In the first case we get 
	$\Box_T \left(\phi_b(\num{x}) = 0\right)$ 
	using provable $\Sigma_1$-completeness,
	whence $\Box_T B(\num{x})$ follows. In the second case we also use provable $\Sigma_1$-completeness to get 
	$$
	\Box_T \left(
	s(F(\Delta^S_\alpha(\num{x}))) = \num{u} \land \Phi_b(\num{x}) \leqslant \num{u} \land \phi_b(\num{x}) = \num{y} \land \num{y} \neq 0
	\right),
	$$
	which implies $\Box_T\neg B(\num{x})$. Now assume $\Phi_b(\num{x}) > u$. Again, using provable $\Sigma_1$-completeness we obtain
	$\Box_T \left(
	s(F(\Delta^S_\alpha(\num{x}))) = \num{u} \land 
	\Phi_b(\num{x}) > \num{u}
	\right)$, whence 
	$$
	\Box_T \forall z \left( s(F(\Delta^S_\alpha(\num{x}))) = z \to \Phi_b(\num{x}) > z\right),
	$$
	which implies $\Box_T \neg B(\num{x})$. Here we rely on the fact that
	$$
	\Box_T \forall z \left(s(F(\Delta^S_\alpha(\num{x}))) = z \to z = \num{u} \right),
	$$
	since we have $\Box_T \left(s(F(\Delta^S_\alpha(\num{x}))) = \num{u}\right)$ and by $\EA$-provable uniqueness of the witness to the formula
	$\exists u \left(s(F(\Delta^S_\alpha(x))) = u \right)$, which was discussed above when it has been defined.

	Thus, we have the following 
	$$
	T \vdash \forall x \left(\Box_T \left(\sharp^x_T \to A \right) \lor \Box_T \left(\sharp^x_T \to \neg A \right) \to 
	\Box^{(k)}_T \left(\Box_T\left(\sharp^x_T \to B\right) \lor \Box_T \left(\sharp^x_T \to \neg B\right) \right)\right).
	$$
	In particular, for each $n \in \nat$ we have
	$$
	T \vdash \Box_T \left(\sharp^n_T \to A \right) \lor \Box_T \left(\sharp^n_T \to \neg A \right) \to 
	\Box^{(k)}_T \left(\Box_T\left(\sharp^n_T \to B\right) \lor \Box_T \left(\sharp^n_T \to \neg B\right) \right).
	$$
	Applying the homomorphism, for each $n \in \nat$ we obtain
	$$
	S \vdash \Box_S \left(\sharp^n_S \to \alpha \right) \lor \Box_S \left(\sharp^n_S \to \neg \alpha \right) \to 
	\Box^{(k)}_S \left(\Box_S\left(\sharp^n_S \to \beta\right) \lor \Box_S \left(\sharp^n_S \to \neg \beta\right) \right).
	$$
	Denote by $j(x)$ the total (by the above) recursive (since $S$ is r.e.) function which maps $n$ to the g{\"o}delnumber of the least $S$-proof of the above implication for this $n$. For each $n \in \nat$ we have
	$$
	S \vdash_{j(n)} \Box_S \left(\sharp^n_S \to \alpha \right) \lor \Box_S \left(\sharp^n_S \to \neg \alpha \right) \to 
	\Box^{(k)}_S \left(\Box_S\left(\sharp^n_S \to \beta\right) \lor \Box_S \left(\sharp^n_S \to \neg \beta\right) \right).
	$$
	%Our aim is to obtain an upper bound (on some infinite set) for %$\Delta^S_\beta$ in terms of $\Delta^S_\alpha$.
	%An $S$-proof $\Box_S\left(\sharp^n_S \to \beta\right) \lor \Box_S %\left(\sharp^n_S \to \neg \beta\right)$ can be obtained by concatenating %the proof of the above implication together with the proof of its %premise, applying modus ponens and then getting rid of $\Box^{(k)}_S$ by %applying the function $\sR_S(x)$ $k$ times. 
	Since $j(x)$ is total, the set  $\{n \in X \mid j(n) > \Delta^S_\alpha(n)\}$ is recursive: given $n$, compute $j(n)$ and search for an $S$-proof of $\sharp^n_S \to \alpha$ or $\sharp^n_S \to \neg\alpha$ with g{\"o}delnumber less than $j(n)$.
	It follows that the set $Y = \{n \in X \mid j(n) \leqslant \Delta^S_\alpha(n)\}$ is nonrecursive and, in particular, infinite,
	since, otherwise, the set $X = \{n \in X \mid j(n) > \Delta^S_\alpha(n)\} \cup \{n \in X \mid j(n) \leqslant \Delta^S_\alpha(n)\}$ would be recursive.
	
	Fix an elementary function $u(x)$ given by Lemma \ref{lm:proofbox}.
	Let us estimate the g{\"o}delnumber of an $S$-proof of
	$\Box^{(k)}_S \left(\Box_S\left(\sharp^n_S \to \beta\right) \lor \Box_S \left(\sharp^n_S \to \neg \beta\right) \right)$ for $n \in Y$. The proof can be obtained as follows: concatenate the proof of implication and its antecedent and apply modus ponens. Let $m(x, y)$ be an elementary function which, given the upper bounds $x$ and $y$ on the proofs of implication and its antecedent, outputs an upper bound on a g{\"o}delnumber of a proof of the consequent. We may assume $m(x, y)$ to be monotone in both arguments. The g{\"o}delnumber of the required $S$-proof is then bounded by
	$$
	m(j(n), u(\Delta^S_\alpha(n))) \leqslant m(\Delta^S_\alpha(n), u(\Delta^S_\alpha(n))) \leqslant r(\Delta^S_\alpha(n)),
	$$ 
	since $j(n) \leqslant \Delta^S_\alpha(n)$ for $n \in Y$.
	Here $r(x)$ is some strictly monotone elementary bound on $m(x, u(x))$
	(e.g., $r(x) = 2^x_l$ for some fixed $l \in \nat$). It is important to note that the function $r(x)$ does not depend on $F(x)$ (and can be fixed ahead of the proof).
	For each $n \in Y$ we then have
	$$
	S \vdash_{r(\Delta^S_\alpha(n))} \Box^{(k)}_S \left(\Box_S\left(\sharp^n_S \to \beta\right) \lor \Box_S \left(\sharp^n_S \to \neg \beta\right) \right),
	$$
	whence  
	$
	S \vdash_{\sR^{(k)}_S(r(\Delta^S_\alpha(n)))} \Box_S\left(\sharp^n_S \to \beta\right) \lor \Box_S \left(\sharp^n_S \to \neg \beta\right).
	$
	Let $v(x)$ be an elementary function given by Lemma \ref{lm:disjR}.
	It follows that for each $n \in Y$ we have 
	$$
	\Delta^S_\beta(n) \leqslant \sR_S(v(\sR^{(k)}_S(r(\Delta^S_\alpha(n))))),
	$$
	so $\Delta^S_\beta \preccurlyeq_Y \sR_S \circ 
	v \circ \sR^{(k)}_S \circ r$.
	Finally, combining this inequality with $F \circ \Delta^S_\alpha \preccurlyeq_X \Delta^S_\beta$ (which has been proved above) we obtain
	$$
	F \circ \Delta^S_\alpha \preccurlyeq_X \Delta^S_\beta \preccurlyeq_Y
	\sR_S \circ v \circ \sR^{(k)}_S \circ r \circ \Delta^S_\alpha.
	$$	
	Fix an elementary function $t(x)$ (this is the only function on the right-hand side of the following inequality which depends on $F(x)$) such that 
	$$
	F \circ \Delta^S_\alpha \leqslant_Y t \circ \sR_S \circ v \circ \sR^{(k)}_S \circ r \circ \Delta^S_\alpha.
	$$
	It follows that 
	$F \leqslant_Z t \circ \sR_S \circ v \circ \sR^{(k)}_S \circ r$, 
	where $Z = \{ \Delta^S_\alpha(n) \mid n \in Y\}$ is an infinite set, since the value $\Delta^S_\alpha(n)$ is a g{\"o}delnumber of an $S$-proof of a certain sentence depending on $n$, so this proofs are all different for distinct values of $n$.
	
	What is left to show then is that the right-hand side of the above inequality can be dominated by the function $\sR^{(k+1)}_S \circ q$ for some fixed elementary function $q(x)$ (which does not depend on the function $F(x)$).
	Let $h(x)$ be an elementary function given by Corollary \ref{cor:fRltRc}.
	We have 
	$t(\sR_S(y)) \leqslant \sR_S(h(y))$
	for all sufficiently large $y$, whence
	$$
	t(\sR_S(v(\sR^{(k)}_S(r(n))))) \leqslant 
	\sR_S(h(v(\sR^{(k)}_S(r(n)))))
	$$
	for all sufficiently large $x$, since $v(\sR^{(k)}_S(r(x)))$ is monotone and unbounded (as a composition of such functions). Define the function $q(x)$ to be $h(v(r(x)))$. If $k = 0$, we are done since the right-hand side of the inequality above is $\sR_S(q(x))$.
	If $k \geqslant 1$, using the same corollary we get
	$$
	h(v(\sR^{(k)}_S(r(n)))) \leqslant \sR^{(k)}_S(h(r(n)))
	$$
	for all sufficiently large $n$, since $r(x)$ is strictly monotone.
	Combining both inequalities above together with monotonicity of $\sR_S(x)$ and strict monotonicity of $h(x), v(x)$ and $r(x)$ we get
	$$
	t(\sR_S(v(\sR^{(k)}_S(r(n))))) \leqslant 
	%\sR_S(h(v(\sR^{(k)}_S(r(n))))) \leqslant
	\sR^{(k+1)}_S(h(r(n))) \leqslant 
	\sR^{(k+1)}_S(h(v(r(n)))) = \sR^{(k+1)}_S(q(n)),
	$$
	for all sufficiently large $n$, as required. 
	
	Finally, we have
	$\sR^{(k+1)}_S(q(n)) \in \cF^{k+1}_S(\EA)$ by Proposition~\ref{prop:boxkRk}.	
\end{proof}

The following corollary contains several sufficient conditions for the non-existence of an epimorphism from $\fD_T$ onto $\fD_S$, which are more convenient to use in applications of Theorem~\ref{th:main}.
\begin{corollary}\label{cor:RS2}
	Each of the following conditions implies that there are no epimorphisms from $\fD_T$ onto $\fD_S$,
	\begin{enumerate}
		\item [$(i)$] 
		$\sR^{(k+1)}_S(x) \in \cF^k_{T}(T)$ for some $k \in \nat$;

		\item [$(ii)$] $\cF^{k+1}_{S}(\EA) \subseteq \cF^k_{T}(T)$ for some $k \in \nat$;
		
		\item [$(iii)$] $T \vdash \forall \sigma \in \Sigma_1 \left(\Box^{(k+1)}_S \sigma \to \Box^{(k)}_T \sigma \right)$ for some $k \geqslant 1$, or $T \vdash \forall \sigma \in \Sigma_1 \left(\Box_S \sigma \to \mathsf{Tr}_{\Sigma_1}(\sigma) \right)$.
	\end{enumerate}
\end{corollary}
\begin{proof}
	Firstly, we show that condition $(i)$ implies that there are no
	epimorphisms from $\fD_T$ onto $\fD_S$. Assuming there is one,
	let $q(x)$ be an elementary function given by Theorem~\ref{th:main}.
	Condition $(i)$ implies that the function $F(x) = \sR^{(k+1)}_S(q(x)) + 1$ also belongs to $\cF^k_{T}(T)$, whence, by Theorem \ref{th:main}, it cannot dominate the function $\sR^{(k+1)}_S(q(x))$, a contradiction.
	Condition $(ii)$ implies $(i)$ since,
	by Proposition \ref{prop:boxkRk},
	$\sR^{(k+1)}_S(x) \in \cF^{k+1}_{S}(\EA) \subseteq \cF^k_{T}(T)$.
	Finally, $(ii)$ follows from $(iii)$ by the definition of the classes 
	$\cF^{k+1}_{S}(\EA)$ and $\cF^k_{T}(T)$.
\end{proof}

{\em Remark.} Points $(ii)$ and $(iii)$ of Corollary \ref{cor:RS2} imply the non-existence of an epimorphism from $(\fL_T, \Box'_T)$ onto $(\fL_S, \Box'_S)$ even when predicates $\Box'_T$ and $\Box'_S$ are not standard, but may be replaced with $\EA$-provably equivalent standard predicates $\Box_T$ and $\Box_S$. Indeed, in this case points $(ii)$ and $(iii)$ also hold for $\Box_T$ and $\Box_S$, whence there are no epimorphisms from $(\fL_T, \Box_T)$ onto $(\fL_S, \Box_S)$. But $(\fL_T, \Box'_T) = (\fL_T, \Box_T)$ and $(\fL_S, \Box'_S) = (\fL_S, \Box_S)$ (we have pure equality here rather than just isomorphism, since boxes induce the same operators by being provably equivalent), and the result follows for the diagonalizable algebras with non-standard provability predicates. Further we heavily rely on this observation when applying Corollary~\ref{cor:RS2}.

\medskip

As can be seen from the proof of Theorem \ref{th:main}, it is quite hard to explicitly specify an algebraic property of diagonalizable algebras which ensures non-isomorphism.
Nevertheless, Shavrukov has shown that under the assumptions of Theorem \ref{th:shavnoniso}, 
the algebras $\fD_T$ and $\fD_S$ are not elementarily equivalent. 
The distinguishing property and a corresponding formula (in the language of diagonalizable algebras) suggested by Shavrukov, however, is not algebraic in nature and is, in fact, a translation of a certain arithmetical fact about partial recursive functions and their indexing by means of a formal theory (cf. Lemma \ref{lm:index}) into the language of diagonalizable algebras (see \cite[Theorem 2.11]{Shav97b}). The construction of such a translation (the interpretation of $\nat$ inside $\fD_T$) is highly non-trivial and even more technical than the proof of non-isomorphism theorem.  
As mentioned in \cite{Adam11} (see the discussion after \cite[Theorem 25]{Adam11}) the modified proof of Theorem \ref{th:adamnoniso} also yields $\fD_T \not \equiv \fD_S$ via the same first-order formula as in \cite{Shav97b}.
Let us check that our further modification also yields the result. Namely, we have the following.

\begin{corollary}\label{cor:RS2eleq}
Any of the conditions $(i)$, $(ii)$ or $(iii)$ of Corollary \ref{cor:RS2}
implies that $\fD_T$ and $\fD_S$ are not elementarily equivalent.
\end{corollary}
\begin{proof}
We only sketch the proof, as it almost the same as that of \cite[Theorem 2.11]{Shav97b}. As in the proof of Corollary \ref{cor:RS2} we see that under the assumption that $(i)$ holds
we get that  $F(x)  \in \cF^k_{T}(T)$, where $F(x) = \sR^{(k+1)}_S(q(x)) + 1$ and $q(x)$ is given by Theorem \ref{th:main} (and may be assumed to be cumulative and monotone).

The first-order sentence in the language of diagonalizable algebras which holds in $\fD_T$ but fails in $\fD_S$ is the translation of the following arithmetical assertion (see \cite[Theorem 2.11]{Shav97b} for more details) (A): {\it for every $\delta$-index $\alpha$ of $h(x)$ there is a $\delta$-index for an extension of $f(x)$ such that for each $n \in \nat$ it is provable (in the theory under consideration) that}
$$
\Box \left(\sharp^n \to \alpha \right) \lor \Box \left(\sharp^n \to \neg \alpha \right) \to 
\Box^{(k)} \left(\Box\left(\sharp^n \to \beta\right) \lor \Box \left(\sharp^n \to \neg \beta\right) \right).
$$
Here  $h(x)$ and $f(x)$ are $\{0, 1\}$-valued partial recursive functions fixed in advance as follows. The function $h(x)$ is taken to be $\phi_c$ from the proof of Theorem \ref{th:main}, which is obtained by applying Lemma \ref{lm:BlumX} to a fixed r.e.\,nonrecursive set $X$. The function $f(x)$ is obtained by applying Lemma \ref{lm:BlumPhi} to the cumulative function $F \circ \Delta^S_\alpha$, where $\alpha$ is a fixed $\delta^S$ index for $\phi_c$ obtained via Lemma~\ref{lm:index}.

To check that (A) fails in $\fD_S$ one just goes through the second part of the proof of Theorem~\ref{th:main}. Indeed, aiming at a contradiction, assume that for the specific $\delta^S$-index $\alpha$ chosen above there is some $\delta^S$-index $\beta$ for an extension of $f$ such that 
	$$
S \vdash \Box_S \left(\sharp^n_S \to \alpha \right) \lor \Box_S \left(\sharp^n_S \to \neg \alpha \right) \to 
\Box^{(k)}_S \left(\Box_S\left(\sharp^n_S \to \beta\right) \lor \Box_S \left(\sharp^n_S \to \neg \beta\right) \right).
$$
for each $n \in \nat$. The proof of Theorem \ref{th:main} then shows that $F(n) \leqslant \sR^{k+1}_S(q(n))$ for infinitely many $n \in \nat$, which contradicts our choice of $F(x)$.

To check that (A) holds in $\fD_T$ one uses the same reasoning as in the first part of the proof of Theorem \ref{th:main}. Namely, given an arbitrary $\delta^T$-index $A$ of $h$ we construct a $\delta^T$-index $B$ of the function $f$ exactly as above, the only difference being the following. The formula $B(x)$ used in the construction of $B$ is now defined as 
$$
B(x) \leftrightharpoons \exists u \left(s(F(p(\Delta^T_A(x)))) = u \land \Phi_b(x) \leqslant u \right) \to \phi_b(x) = 0,
$$
where $s(x)$ is a monotone elementary function with $\Phi_b \leqslant_X s \circ F \circ \Delta^S_\alpha$ and $p(x)$ is a monotone elementary function with such that $\Delta^S_\alpha \leqslant_X p \circ \Delta^T_A$. The monotonicity of $F(x)$ then yields 
$$\Phi_b \leqslant_X s \circ F \circ \Delta^S_\alpha \leqslant_X s \circ F \circ p \circ \Delta^T_A,
$$
which is sufficient for the rest of the proof to go through.
%In the above proof the index $A$ was a preimage of $\alpha$ under an epimorphism and the function $f(x)$ (named $\phi_b$ in the proof) was defined in terms of a given program . Here $A$ is arbitrary, but we still use the same formula $\psi(x, y)$ to represent the function $\Delta^S_\alpha$ in the construction of the formula $B(x)$ in the above proof. The proof 
  
\end{proof}

Although we improved the inequality obtained by Adamsson (cf. Theorem \ref{th:adamnoniso}), one may still ask about the possible strengthening of Theorem \ref{th:main} by replacing $k+1$ on the $S$-side with~$k$. Namely,
by replacing the class $\cF^{k+1}_{S}(\EA)$ with a more natural
class $\cF^k_{S}(S)$, which would be completely symmetric to
the class $\cF^k_{T}(T)$ present in the statement of the theorem. In the current formulation this is simply not true: take $T = S$ and $\Box_T = \Box_S$, then every function in $\cF^k_{S}(S)$ is dominated by one in $\cF^k_{T}(T) = \cF^k_{S}(S)$. But we may weaken the conclusion of the theorem by swapping the quantifiers to the following one: no function in $\cF^k_{T}(T)$ dominates each function in $\cF^{k+1}_{S}(\EA)$. Can we replace $\cF^{k+1}_{S}(\EA)$ with $\cF^k_{S}(S)$ in this weakened formulation? Below we show that this is not true for $T = \EA^+ + \RFN_{\Sigma_1}(\EA^+)$ 
and $S = \EA^+ + \Rfn_{\Sigma_1}(\EA^+)$, which have isomorphic diagonalizable algebras by Theorem \ref{th:shaviso} (being $\EA^+$-provably $\Sigma_2$-equivalent), but $\cF(T) = \mathcal{E}_5$ contains a function dominating each function in $\cF(S) = \mathcal{E}^4$ (and the same also holds for $\cF^k_{T}(T)$ and $\cF^k_{S}(S)$ for each $k \in \nat$).

A plausible hypothesis is that the existence of an epimorphism from $\fD_T$ onto $\fD_S$ implies that 
no function from $\cF_{T}(T)$ dominates each function from 
$\cF_{S}(T)$ (and analogously for $k > 1$). This condition holds for all known examples of theories with isomorphic algebras and is related to Shavrukov's isomorphism condition (see Theorem \ref{th:shaviso}). In practice, when isomorphism theorem is applied the theories turn out to be (at least) $\mathcal{B}(\Sigma_1)$-coherent. But if $T$ and $S$ are $\Sigma_1$-coherent (or even $\Sigma_1$-equivalent provably in $T$), then we have $\cF_{T}(T) = \cF_{S}(T)$
by definition of these classes, whence the condition holds. It seems that to find counterexamples to this claim would require new methods for establishing isomorphisms of diagonalizable algebras. Another hint that this claim may be true lies in the proof of the main theorem. Specifically, the place where we first prove (we consider the case $k = 1$)
	$$
T \vdash \forall x \left(\Box_T \left(\sharp^x_T \to A \right) \lor \Box_T \left(\sharp^x_T \to \neg A \right) \to 
\Box_T \left(\Box_T\left(\sharp^x_T \to B\right) \lor \Box_T \left(\sharp^x_T \to \neg B\right) \right)\right),
$$
and then weaken this to 
$$
T \vdash \Box_T \left(\sharp^n_T \to A \right) \lor \Box_T \left(\sharp^n_T \to \neg A \right) \to 
\Box_T \left(\Box_T\left(\sharp^n_T \to B\right) \lor \Box_T \left(\sharp^n_T \to \neg B\right) \right),
$$
for each $n \in \nat$. The first derivation establishes the totality of the function $f(x)$ such that $f(\Delta^T_A(x))$ gives the least $T$-proof of $\exists y \left(\Delta^T_B(\num{x}) = y\right)$. It seems that not using this fact (or some variation of it) and just instantiating this derivation for each $n \in \nat$ leads to the significant loss of information. And we also have that each function in $\cF_{S}(T)$ is dominated by one of the form $\sR_S(f(x))$ for some $f(x) \in \cF(T)$ (cf. Proposition \ref{prop:Fboxdef}).

\begin{question}
	Does the existence of an epimorphism from 
$\fD_T$ onto $\fD_S$ imply that 
no function from $\cF_{T}(T)$ dominates each function from 
$\cF_{S}(T)$?
\end{question} 

Another question which is also relevant to the above discussion is the following.

\begin{question}
Can ``for infinitely many $n \in \nat$'' be replaced with a more natural ``for almost all $n \in \nat$'' in the inequality of the main theorem?
\end{question}

Now, we would like to apply Theorem \ref{th:main} to obtain new examples of theories with non-isomorphic diagonalizable algebras.
In practice we apply Corollary \ref{cor:RS2} and also rely on the remark after it, when the provability predicate under consideration is not standard (e.g., $\Box_T\Box_T$), i.e.,
in such cases we prove that the non-standard provability predicate is $\EA$-provably equivalent to a standard one (note that in \cite{Adam11}
this issue is not discussed in details).

Firstly, let us obtain the result about non-isomorphism of 
$(\fL_T, \Box_T)$ and $(\fL_T, \Box_T\Box_T)$
anticipated by Adamsson in \cite{Adam11}, and a more general result of this kind.
We start with the following lemma, which is needed to apply the remark after Corollary \ref{cor:RS2}.
\begin{lemma}\label{lm:boxkstand}
	For each $k \geqslant 2$  there is a standard provability predicate 
	which is equivalent to $\Box^{(k)}_T$ provably in $\EA$.
\end{lemma}	
\begin{proof}
Let us consider the case $k = 2$ (the case $k > 2$ being completely analogous).
We construct an elementary axiomatization $\sigma^2_T(x)$ of $T$ such that 
$$
\EA \vdash \forall \psi \left(
\Box_T\Box_T \psi \eqv \Box_{\sigma^2_T}\psi
\right),
$$	
where $\Box_{\sigma^2_T}\psi$ is $\exists p\, \Prf_{\sigma^2_T}(p, \gn{\psi})$.
Define $\sigma^2_T(x) := \exists p, \theta \leqslant x \left(
x = \gn{\num{p} = \num{p} \to \theta} \land \Prf_T(p, \gn{\Box_T\theta})
\right)$, where $\Prf_T(p, x)$ is a standard proof predicate for (a fixed elementary axiomatization of) $T$.
Firstly, we have
\begin{align*}
\EA \vdash \Box_T\Box_T \psi
&\to \exists p\, \Prf_T(p, \gn{\Box_T \psi})\\
&\to \exists p\, \sigma^2_T(\gn{\num{p} = \num{p} \to \psi})\\
&\to \exists p\, \Box_{\sigma^2_T}\left(\num{p} = \num{p} \to \psi\right)\\
&\to \Box_{\sigma^2_T}\psi,
\end{align*}
since, certainly, $\EA \vdash \forall p\, \Box_{\sigma^2_T}\left(\num{p} = \num{p}\right)$. 

To show the converse we argue as follows. Let us show that we have $\Box_T\Box_T\psi$ for each $\sigma^2_T$-axiom $\psi$ of $T$. Fix $p$ and $\theta$ such that $\psi \leftrightharpoons (\num{p} = \num{p} \to \theta)$ and $\Prf_T(p, \gn{\Box_T\theta})$. The latter implies $\Box_T\Box_T\theta$, whence $\Box_T\Box_T \left(\num{p} = \num{p} \to \theta\right)$, i.e., $\Box_T\Box_T \psi$,  by logic. Formalizing this in $\EA$ we get
$$
\EA \vdash \forall \psi \left(
\sigma^2_T(\gn{\psi}) \to \Box_T\Box_T\psi
\right).
$$
As usual, we use $\Pi_2$-conservativity of the $\Sigma_1$-collection schema $\mathsf{B}\Sigma_1$ over $\EA$ to strengthen this~to
$$
\EA \vdash \forall \psi \left(
\Box_{\sigma^2_T}\psi \to \Box_T\Box_T\psi
\right).
$$
Indeed, due to the formalized deduction theorem it is sufficient to show that
$$
\EA \vdash \forall z\, \forall \psi \left( \forall i < z\, \sigma^2_T(\gn{\psi_i})
\to \Box_T\Box_T \bigwedge_{i < z} \psi_i
\right).
$$
We have $\EA \vdash \forall z\, \forall \psi \left( \forall i < z\, \sigma^2_T(\gn{\psi_i}) \to 
\forall i < z\, \Box_T\Box_T\psi_i\right)$. 
Using $\Sigma_1$-collection we obtain 
$$
\EA + \mathsf{B}\Sigma_1 \vdash 
\forall z\, \forall \psi \left(
\forall i < z\, \Box_T\Box_T\psi_i
\to 
\Box_T \left(\bigwedge_{i < z} \Box_T\psi_i\right)
\right).
$$
Finally, using induction on $z$ (and distributivity of $\Box_T$ over conjunction), 
which is formalizable in $\EA$ since we have an elementary bound on the corresponding $T$-proof, we get
$$
\EA \vdash \forall z\, \forall \psi\,
\Box_T  \left( \left(\bigwedge_{i < z} \Box_T\psi_i\right) \to
\Box_T \bigwedge_{i < z} \psi_i \right).
$$
Combining the three derivations above we get
$$
\EA + \mathsf{B}\Sigma_1 \vdash \forall z\, \forall \psi \left( \forall i < z\, \sigma^2_T(\gn{\psi_i})
\to \Box_T\Box_T \bigwedge_{i < z} \psi_i
\right),
$$
whence the same $\Pi_2$-sentence is provable in $\EA$, as required.
\end{proof}

\begin{corollary}\label{cor:boxsq}
There is no epimorphism from $(\fL_T, \Box_T \Box_T)$ onto $(\fL_T, \Box_T)$.
\end{corollary}
\begin{proof}
The result follows by point $(iii)$ of Corollary \ref{cor:RS2} with $k = 1$ since, clearly,
$$
\EA \vdash \forall \sigma \in \Sigma_1 \left(
\Box_T \Box_T \sigma \to \Box_T \Box_T \sigma
\right).
$$
\end{proof}

More generally, we have the following result.
\begin{corollary}\label{cor:boxmn}
For each $n, m \in \nat$ if $n > m \geqslant 1$, then there are no epimorphisms from $(\fL_T, \Box^{(n)}_T)$ onto $(\fL_T, \Box^{(m)}_T)$.
\end{corollary}
\begin{proof}
Using provable $\Sigma_1$-completeness and $m + 1 \leqslant n$ we get
$$
\EA \vdash \forall \sigma \in \Sigma_1 \left(
\Box^{((m+1)\cdot m)} \sigma \to \Box^{(m \cdot n)}_T \sigma \right),
$$
whence the result follows by point $(iii)$ of Corollary \ref{cor:RS2} with $k = m$.
\end{proof}

Let us note that there is a more algebraic proof of the Corollary \ref{cor:boxsq} for which the improved upper bound is not needed. It follows already from Adamsson's results. Indeed, Adamsson showed that there are no epimorphisms from $(\fL_T, \Box^{(8)}_T)$ onto $(\fL_T, \Box_T)$ (he used $\Box^{(6)}_T$ but any larger iteration also works). Assuming the existence of an epimorphism $j \colon (\fL_T, \Box_T\Box_T) \to (\fL_T, \Box_T)$, one can check that $j \circ j \circ j$ would be an epimorphism from $(\fL_T, \Box^{(8)}_T)$ onto $(\fL_T, \Box_T)$, a contradiction. In general, if $j \colon (\fL_T, \Box^{(n)}_T) \to (\fL_T, \Box^{(m)}_T)$ is an epimorphism with $n > m \geqslant 1$, then $j^{(k)}$ is an epimorphism from $(\fL_T, \Box^{(n^k)}_T)$ onto $(\fL_T, \Box^{(m^k)}_T)$ for each $k \geqslant 1$, which leads to the contradiction, since for $n > m$ we have $n^k > 6 \cdot m^k$ for some (sufficiently large) $k \in \nat$. The number $6$ here corresponds to Adamsson's gap between iterated boxes for which there are no such epimorphisms. In this way one can also obtain a proof of Corollary \ref{cor:boxmn}, we omit the details.

Although Shavrukov's non-isomorphism condition (Theorem \ref{th:shavnoniso}) is strictly stronger (so the result is weaker) than that of Adamsson (Theorem \ref{th:adamnoniso}), it seems to be more proof-theoretic in nature and more applicable to particular pairs of theories. The application of Adamsson's condition requires a more thorough analysis of the size of proofs of $\Sigma_1$-sentences in both theories in order to compare the growth rates of $\sR_T(x)$ and $\sR_S(x)$, yet it yields new non-isomorphism results, which are not covered by Theorem \ref{th:shavnoniso}.

The following theorem (as well as point $(iii)$ of Corollary \ref{cor:RS2}) provides a stronger proof-theoretic non-isomorphism condition similar in formulation to that of Shavrukov. As noted in \cite{Bek05}, the condition of Theorem \ref{th:shavnoniso} already follows from a more natural $T \vdash \RFN_{\Sigma_1}(S)$. It can also be proved more directly as follows. By Proposition \ref{prop:boxkRk} we have $\EA\vdash \forall x\, \Box_S \exists y\,( \sR_S(\num{x}) = y)$, where $\sR_S(x) = y$ is a $\Sigma_1$-formula defining the graph of $\sR_S(x)$, constructed in the proof of this proposition. Applying $\RFN_{\Sigma_1}(S)$, which is provable in $T$, we get $T \vdash \forall x\, \exists y \left(\sR_S(x) = y\right)$, that, by definition of $\sR_S(x)$, certainly implies the required form of reflection in Theorem~\ref{th:shavnoniso}. It turns out that one can replace uniform $\Sigma_1$-reflection with the local one as the following result shows (note that since the schema $\Rfn_{\Sigma_1}(S)$, unlike $\RFN_{\Sigma_1}(S)$, is not finitely axiomatizable, we need to be more precise when claiming its provability in $T$ in the formal context). 

\begin{theorem}\label{th:locrfn}
If $T \vdash \Rfn_{\Sigma_1}(S)$ and this fact is formalizable in the following form
$$
T \vdash \forall \sigma \in \Sigma_1\, \Box_T \left( \Box_S \sigma \to \sigma  \right),
$$
then there are no epimorphisms from $\fD_T$ onto $\fD_S$.
\end{theorem}
\begin{proof}
Applying provable $\Sigma_1$-completeness and the assumption $T \vdash \forall \sigma \in \Sigma_1\, \Box_T \left( \Box_S \sigma \to \sigma \right)$ twice, we obtain
\begin{align*}
T \vdash \forall \sigma \in \Sigma_1 \Bigl( \Box_S \Box_S \sigma \Bigr. &\to \Box_T \Box_S \Box_S \sigma \\
&\to \Box_T\Box_S \sigma\\
&\to \Bigl. \Box_T\sigma \Bigr). 
\end{align*}
Point $(iii)$ of Corollary \ref{cor:RS2} with $k = 1$ then yields the result.
\end{proof}

As a corollary we get a family of pairs of theories with the same class of provably total computable functions (since the schema of local reflection for $T$ is contained in the extension of $T$ with set of all true $\Pi_1$-sentences) but non-isomorphic diagonalizable algebras.
\begin{corollary}\label{cor:locrfn}
For each $n \geqslant 1$, there are no epimorphisms from $\fD_{T + \Rfn_{\Sigma_1}(T)}$ onto $(\fL_T, \Box^{(n)}_T)$.
\end{corollary}
\begin{proof}
By Theorem \ref{th:locrfn}, since $T + \Rfn_{\Sigma_1}(T)$ proves local $\Sigma_1$-reflection for $T$ with $\Box^{(n)}_T$ (apply reflection $n$ times) and this is formalizable in $\EA$.
\end{proof}

The following proposition provides another example of this kind.
\begin{proposition}
There are theories $U$, $S$ and $T$ such that 
$$
\fD_U \not \cong \fD_S \cong \fD_T, \text{ while }
\cF(U) = \cF(S) \neq \cF(T).
$$
\end{proposition}
\begin{proof}
Consider the following theories
$$
U = \EA^+, \quad S = \EA^+ + \Rfn_{\Sigma_1}(\EA^+), 
\quad T = \EA^+ + \RFN_{\Sigma_1}(\EA^+).
$$
Firstly, Corollary \ref{cor:locrfn} implies that $\fD_U \not \cong \fD_S$.  The isomorphism $\fD_S \cong \fD_T$ follows from Theorem \ref{th:shaviso}, since $T\equiv_{\Sigma_2} S$ (whence also $T \equiv_{\mathcal{B}(\Sigma_1)} S$) provably in $\EA^+$ by \cite[Proposition 5.4]{Bek03}, and $\EA^+$ is closed under $\Sigma_1$-collection rule by \cite[Corollary 5.2]{Bek98}, so $T$ and $S$ are $\mathcal{B}(\Sigma_1)$-coherent.

Now, since $S$ is contained in the extension of $U$ with the set of all true $\Pi_1$-sentences we get $\cF(S) = \cF(U)$. It is known (see, e.g., \cite[Corollary 3.3]{Bek03}) that 
$\cF(S) = \cF(U) = \mathcal{E}_4$ and $\cF(T) = \mathcal{E}_5$, so $\cF(T) \neq \cF(S)$, where $\mathcal{E}_i$ is the $i$th class of the Grzegorczyk hierarchy.
\end{proof}

The example given in the above proposition also demonstrates that it is not possible to strengthen the statement of the main theorem in a way discussed after Corollary \ref{cor:RS2} (so in a certain sense, it is optimal). Namely, if there is an epimorphism from $\fD_T$ onto $\fD_S$, then for each $k \in \nat$, 
no function in $\cF^k_{T}(T)$ dominates each function in $\cF^k_{S}(S)$. Indeed, this is not true already for $k = 0$ (and, in fact, for each $k$) with $T = \EA^+ + \RFN_{\Sigma_1}(\EA^+)$ and $S = \EA^+ + \Rfn_{\Sigma_1}(\EA^+)$, since $\cF^0_{T}(T) = \cF(T) = \mathcal{E}_5$ contains a function which dominates each function in $\mathcal{E}_4 = \cF(S) = \cF^0_{S}(S)$. The case $k = 0$ may seem to be degenerate, so let us also show that this does not hold for $k = 1$. We need to find a function $F(x) \in \cF_{T}(T)$ which dominates each function in $\cF_{S}(S)$. Let $f(x) \in \mathcal{E}_5$ be a function dominating each function in $\mathcal{E}_4$ (e.g., the diagonal of superexponential function $2^x_x$). We claim that $F(x) = \sR_S(f(x))$ has the desired property. Indeed, since $S \subseteq T$ provably in $\EA$ we have $\cF_{S}(S) \subseteq \cF_{T}(T)$, whence by Proposition $\ref{prop:boxkRk}$ we have $\sR_S(x) \in \cF_{T}(T)$. But then, since $f(x) \in \cF(T)$, we also have $\sR_S(f(x)) \in \cF_{T}(T)$ (see the proof of Lemma \ref{lm:Fkcompos} for such type of argument), as required. Now, given any $G(x) \in \cF_{S}(S)$, let $g(x)$ be an $S$-provably total computable function which gives the (g{\"o}delnumber of the) least $S$-proof of $\exists y\, \psi_G(\num{x}, y)$, where $\psi_G(x, y)$ is a $\Sigma_1$-definition of the graph of $G(x)$. By definition of $\sR_S(x)$, we then have $G(x) \leqslant \sR_S(g(x))$, which is dominated by $\sR_S(f(x))$ by the choice of $f(x)$, since $g(x) \in \cF(S) = \mathcal{E}_4$, and monotonicity of $\sR_S(x)$.

Note that in all examples of theories $T$ and $S$ with non-isomorphic diagonalizable algebras considered so far the theories were not provably equiconsistent (in both $T$ and $S$).
The next theorem gives a natural family of $\EA$-provably $\Pi_1$-equivalent whence also $\Pi_1$-coherent theories (cf.~Theorem \ref{th:shaviso}), since $\EA$ is closed under $\Sigma_1$-collection rule by \cite[Lemma 4.1]{Bek98},  with non-isomorphic algebras. Recall that the theory $T_\omega$ is defined as $T$ together with all its finite iterated consistency assertions, i.e., sentences $\neg \Box^{(n)}_T\bot$ for all $n \in \nat$. By Goryachev's theorem we have $T_\omega \equiv_{\Pi_1} T + \Rfn_{\Sigma_1}(T)$, and this holds provably in $\EA$ (see \cite[Proposition 6.1]{Bek03}).
\begin{theorem}\label{th:omegalocrfn}
	There is no epimorphism from $\fD_{T + \Rfn_{\Sigma_1}(T)}$ onto $\fD_{T_\omega}$.
\end{theorem}
\begin{proof}
We apply point $(i)$ of Corollary \ref{cor:RS2}. By Lemma \ref{lm:pi1extF} and Proposition \ref{prop:boxkRk} we have
$$
\sR_{T_\omega}(x) \in \cF_{{T_\omega}}(\EA) = \cF_{T}(\EA),
$$
whence, by Lemma \ref{lm:Fkcompos}, $\sR^{(2)} _{T_\omega}(x) \in \cF^2_{T}(\EA) \subseteq \cF_{{T + \Rfn_{\Sigma_1}(T)}}(T + \Rfn_{\Sigma_1}(T))$, where the inclusion is due to
$$
\EA \vdash \forall \sigma \in \Sigma_1 \left(
\Box_T\Box_T \sigma \to \Box_{T+ \Rfn_{\Sigma_1}} \sigma
\right).
$$
\end{proof}

Since the non-isomorphism condition is formulated in terms of the reflection function (and the related classes) one may ask about the relationship between such functions for various extensions of a given theory $T$. E.g., in the proof of Theorem \ref{th:omegalocrfn} we see 
that $\sR_{T_\omega}(x)$ has the same growth rate (up to an elementary transformation of the input) as $\sR_T(x)$ (and the same holds for any $\Pi_1$-extension of $T$, see Lemma \ref{lm:pi1extF}).
An easy argument using Proposition \ref{prop:boxkRk} shows that $\sR_{T + \Rfn_{\Sigma_1}(T)}(x)$ dominates
each finite iteration of $\sR_T(x)$. The next proposition gives a more precise description of its growth rate.
\begin{proposition}
	The function $\sR_{T + \Rfn_{\Sigma_1}(T)}(x)$ has the same rate of growth as the diagonal function $\sR^{(x+1)}_T(x)$ 
	up to an elementary transformation of the input.
\end{proposition}
\begin{proof}
Firstly, we show that $\sR^{(x+1)}_T(x) \leqslant 
\sR_{T + \Rfn_{\Sigma_1}(T)}(q(x))$ for some elementary function $q(x)$.
We use the fact that the proof of Proposition \ref{prop:boxkRk} is formalizable in $\EA$, i.e.,
$$
\EA \vdash \forall z\, \Box_\EA \left(
\forall x\, \Box^{(z)}_T \exists y \left( 
\sR^{(z)}_T(\num{x}) = y
\right)
\right),
$$
where for each $z$ the formula $\sR^{(z)}_T(\num{x}) = y$ is constructed in a most natural way (to represent $z$-fold composition of $\sR_T(x)$ with itself) using the formula $\psi(x, y)$ defining the graph of $\sR_T(x)$. It follows that
\begin{align*}
\EA \vdash \forall z\, \Box_\EA \left(
\forall x\, \Box^{(z)}_T \exists y \left( 
\sR^{(z)}_T(\num{x}) = y
\right)
\right) 
&\to \forall z\,\Box_\EA \Box^{(z+1)}_T \exists y \left( 
\sR^{(z+1)}_T(\num{z}) = y
\right)\\
&\to \forall z\,\Box_{T + \Rfn_{\Sigma_1}(T)} 
\Box^{(z+1)}_T \exists y \left( 
\sR^{(z+1)}_T(\num{z}) = y
\right)\\
&\to \forall z\,\Box_{T + \Rfn_{\Sigma_1}(T)} 
\exists y \left( 
\sR^{(z+1)}_T(\num{z}) = y
\right),
\end{align*}
whence $\EA \vdash  \forall x\,\Box_{T + \Rfn_{\Sigma_1}(T)} 
\exists y \left( 
\sR^{(x+1)}_T(\num{x}) = y
\right)$, so $T + \Rfn_{\Sigma_1}(T) \vdash_{q(x)} \exists y \left( 
\sR^{(x+1)}_T(\num{x}) = y
\right)$ for some elementary function $q(x)$ (being provably total in $\EA$). Consequently, we obtain $\sR^{(x+1)}_T(x) \leqslant 
\sR_{T + \Rfn_{\Sigma_1}(T)}(q(x))$, as required.

Now, we show that  $\sR_{T + \Rfn_{\Sigma_1}(T)}(x) \leqslant \sR^{(r(x)+1)}_T(r(x))$ for some elementary function $r(x)$. To obtain this we prove that there is a monotone elementary function $r(x)$ which, given a (g{\"o}delnumber of a) $T + \Rfn_{\Sigma_1}(T)$-proof $x$ of a $\Sigma_1$-sentence $\sigma$, transforms it into a $T$-proof $r(x)$ of $\Box^{z}_T\sigma$ for some $z < x$. The required inequality then follows by the definition and monotonicity properties of $\sR_T(x)$.
We prove that $\EA \vdash \forall z\, A(z)$, where 
$$
A(z) \leftrightharpoons 
\forall \psi\, \forall \sigma \in \Sigma_1 \left(
\Box_T \left(\bigwedge_{i < z} \left(\Box_T \psi_i \to \psi_i \right) 
\to \sigma 
\right) 
\to \Box_T \Box^{(z)}_T\sigma
\right).
$$
The proof goes by reflexive induction on $z$, i.e., we show that
$$
\EA \vdash A(0) \land \forall z \left(\Box_\EA A(\num{z}) \to A(z + 1)\right),
$$
whence $\EA \vdash \Box_\EA \forall z\, A(z) \to \forall z\, A(z)$ follows, therefore $\EA \vdash \forall z\, A(z)$, by L{\"o}b's theorem.

We argue informally in $\EA$. Clearly, we have $A(0)$. Now, assuming $\EA \vdash A(\num{z})$ we derive $A(z+1)$. Fix an arbitrary $\psi$ and $\sigma$ such that
$$
T \vdash \bigwedge_{i < z + 1} \left(\Box_T \psi_i \to \psi_i \right) 
\to \sigma.
$$
By $\Sigma_1$-completeness we get $T \vdash \Box_T \left( \bigwedge_{i < z + 1} \left(\Box_T \psi_i \to \psi_i \right) 
\to \sigma\right)$.
It follows that for each $j < z + 1$ we have
\begin{align*}
T \vdash \Box_T \psi_j 
&\to \Box_T \left( \psi_j \land \left(\bigwedge_{i < z + 1} \left(\Box_T \psi_i \to \psi_i \right) 
\to \sigma\right)\right)\\
&\to \Box_T \left( \left(\Box_T \psi_j \to \psi_j \right) \land \left(\bigwedge_{i < z + 1} \left(\Box_T \psi_i \to \psi_i \right) 
\to \sigma\right)\right)\\
&\to \Box_T \left(\left(\bigwedge_{i < z + 1, i \neq j} \left(\Box_T \psi_i \to \psi_i \right) 
\to \sigma\right)\right)\\
&\to \Box_T\Box^{(z)}_T\sigma,
\end{align*}
where the last implication uses the reflexive induction hypothesis $\EA \vdash A(\num{z})$.
We also have the following derivation 
\begin{align*}
T \vdash \bigwedge_{i < z + 1} \neg \Box_T \psi_i  
&\to \left(\bigwedge_{i < z + 1} \left(\Box_T \psi_i \to \psi_i \right) \right)\\
&\to \sigma\\
&\to \Box_T\Box^{(z)}_T\sigma,
\end{align*}
where we use provable $\Sigma_1$-completeness for $\sigma \in \Sigma_1$.
Combining this $z + 2$ derivations (by considering $z + 2$ cases) we derive
$T \vdash \Box_T\Box^{(z)}_T\sigma$, i.e., $T \vdash \Box^{(z+1)}_T\sigma$, as required.

Since the deduction theorem can be formalized in $\EA$ there is a strictly monotone elementary  function $t(x)$ such that
$T + \Rfn_{\Sigma_1}(T)\vdash_x \sigma$ implies
$T \vdash_{t(x)}\bigwedge_{i < z} \left(\Box_T \psi_i \to \psi_i\right) \to \sigma$ for some $\Sigma_1$-sentences $\psi_0, \dots, \psi_{z-1}$. From $\EA\vdash \forall z\, A(z)$ we obtain a strictly monotone elementary function $s(y)$ such that 
$T \vdash_y \bigwedge_{i < z} \left(\Box_T \psi_i \to \psi_i\right) \to \sigma$ implies $T \vdash_{s(y)} \Box^{(z)}_T \sigma$.

Finally, we define $r(x) = s(t(x)) + 1$ (to get $r(x) > x$) and obtain the following 
$$
T + \Rfn_{\Sigma_1}(T)\vdash_x \sigma \Longrightarrow
T \vdash_{r(x)} \Box^{(z)}_T\sigma,
$$
for some $z < x$ (since the number of axioms of $T + \Rfn_{\Sigma_1}(T)$ in the proof is less than its g{\"o}delnumber). Thus, by the definition of $\sR_T(x)$, we get that the least witness to $\sigma$ is bounded by
$\sR^{(z+1)}_T(r(x)) \leqslant \sR^{(r(x)+1)}(r(x))$. The last inequality is obtained as follows. Using Proposition \ref{prop:Rdomprovtot} fix some $c \in \nat$ such that $x \leqslant \sR_T(x)$ for all $x \geqslant c$. It follows that for each $y \geqslant 1$ we have $x \leqslant \sR^{(y)}_T(x)$ for all $x \geqslant c$. Since we may assume $r(x) \geqslant c$ for all $x$ (otherwise, consider $r(x) + c$), and $z < x < r(x)$ 
we get $r(x) \leqslant \sR^{(r(x) + 1 - z)}(r(x))$, whence, by monotonicity of $\sR_T(x)$, the required inequality follows.
\end{proof}
This proposition may be contrasted with \cite[Lemma 3.7]{Bek05}, asserting that for any honest function $f(x)$ we have 
$
\EA \vdash \forall x\, \exists y \left(f^{(x+1)}(x) = y\right) \eqv \RFN_{\Sigma_1}(\EA + \forall x\, \exists y \left(f(x) = y\right)).
$

Let us also note that as an immediate corollary we obtain the equivalence between two times iterated local $\Sigma_1$-reflection for $T$ with standard provability predicate $\Box_T$ and a local $\Sigma_1$-reflection for $T$ with provability predicate $\exists x \Box^{(x)}_T \phi$ (which roughly corresponds to the provability with Parikh's rule).

\section{Bimodal diagonalizable algebras}

As has been already mentioned after the proof of the main theorem
the first-order property demonstrating the lack of elementary equivalence between $\fD_T$ and $\fD_S$ found by Shavrukov is not entirely algebraic in its form and relies on the existence of a quite complicated translation from the language of $\nat$ into that of $\fD_T$ given in \cite{Shav97b}.

By enriching the signature of a diagonalizable algebra with new constants or operators one may hope to find a property distinguishing $\fD_T$ and $\fD_S$, which is more algebraic in nature and provides an easier way to obtain non-isomorphism results.
The work in this direction has been done by Beklemishev \cite{Bek96}, who showed that the diagonalizable algebras of $\EA$ and $\PA$ enriched with constants for (a $\Pi_2$-sentence which axiomatizes) the uniform $\Sigma_1$-reflection schema for a theory are not elementarily equivalent, whence not isomorphic. Beklemishev also showed that $\RFN_{\Sigma_1}(T)$ is definable by a formula of second order propositional provability logic (it is not definable in the language of $\fD_T$ by the results of \cite{Shav97a}, since $\RFN_{\Sigma_1}(T)$ is not equivalent to a boolean combination of $\Sigma_1$-sentences (see, e.g., \cite[Proposition 2.16]{Bek05}))
and constructed a simple example of such formula that distinguishes $\EA$ and $\PA$.
Note that in this example we have $\fD_\EA \not \cong \fD_\PA$ (by Theorem \ref{th:shavnoniso}), so the enriched signature here allows to obtain an easier non-isomorphism proof for expansions, but the reducts (to the language of diagonalizable algebras) are already non-isomorphic.

In this section we enrich the signature of diagonalizable algebras with a new modality $[1]$, which is interpreted as $1$-provability, and consider an expansion of the algebra $\fD_T$ to a bimodal diagonalizable algebra $(\fD_T, [1]_T)$. In our proofs we rely on certain descriptions of the set $\cC_T$ of sentences, which are provably 1-provable in a given theory $T$, obtained in \cite{KolmBek19}. Recall the notation $C_S(T) = \{\phi \mid T \vdash [1]_S \phi \}$ for the set of all sentences which are $T$-provably 1-provable in~$S$.

We provide several examples of pairs of theories which,  unlike in Beklemishev's example, have isomorphic diagonalizable algebras, but non-isomorphic bimodal diagonalizable algebras.
We start with a quite natural pair of theories with such algebras. 

\begin{proposition}\label{prop:isi1pra}
The bimodal algebras $(\fD_{\ISi_1}, [1]_{\ISi_1})$ and $(\fD_{\PRA}, [1]_{\PRA})$ are not isomorphic, 
while the reducts $\fD_{\ISi_1}$ and $\fD_\PRA$ are isomorphic.
\end{proposition}
\begin{proof}
By Parsons' theorem $\ISi_1 \equiv_{\Pi_2} \PRA$, and this result is known to hold provably in $\EA^+$, which is closed under $\Sigma_1$-collection rule by \cite[Corollary 5.2]{Bek98},  
whence these two theories are $\mathcal{B}(\Sigma_1)$-coherent.
Theorem \ref{th:shaviso} then implies $\fD_{\ISi_1} \cong  \fD_\PRA$.

Aiming at a contradiction, assume $(\cdot)^* \colon (\fL_{\ISi_1}, \Box_{\ISi_1}, [1]_{\ISi_1}) \to (\fL_\PRA, \Box_\PRA, [1]_\PRA)$ is an isomorphism.
Consider the filters defined in these algebras by the formula $F(p) := [1] p = \top$.
These are precisely the sets $\cC_{\ISi_1}$ and $\cC_\PRA$ of provably 1-provable sentences (their equivalence classes) in each theory. 
By the results of \cite{KolmBek19} the filter $\cC_\PRA = C_\PRA(\PRA)$
is generated by $\Rfn(\PRA)$ over $\PRA$ (apply \cite[Corollary 3.3]{KolmBek19}, \cite[Lemma 4.1]{KolmBek19} and the axiomatization of $\mathsf{PRA}$ given by \cite[Proposition 3.9]{Bek05}), i.e., by the elements of the form $\Box_\PRA \psi \to \psi$,
while $\cC_{\ISi_1} = C_{\ISi_1}(\ISi_1)$ is not, being generated by $\Rfn(\ISi_1)_\omega$ (see \cite[Corollary 5.3]{KolmBek19}). 
Since the set of all instances of local reflection schema and the filter are definable in the language of bimodal algebra, the generation property must be preserved by an isomorphism, which is a contradiction.

More formally, assume $\phi \in \cC_{\ISi_1}$, i.e., $\ISi_1 \vdash [1]_{\ISi_1} \phi$. By the isomorphism, there is a $\chi$ such that
$\PRA  \vdash [1]_\PRA \chi$ and $(\chi)^* = \phi$. In particular, $\chi \in \cC_\PRA$, so for some $n \in \nat$  there are sentences $\psi_0$, $\dots$, $\psi_{n-1}$, such that
$$
\PRA \vdash \bigwedge_{i < n} (\Box_\PRA \psi_i \to \psi_i) \to \chi.
$$
Applying the isomorphism, we get
$$
\ISi_1 \vdash \bigwedge_{i < n} (\Box_{\ISi_1}  (\psi_i)^* \to (\psi_i)^*) \to \phi,
$$
whence $\ISi_1  + \Rfn(\ISi_1) \vdash \phi$. It follows that $\cC_{\ISi_1 }$ is generated by $\ISi_1  + \Rfn(\ISi_1 )$, which 
contradicts the fact that it is generated by $\Rfn(\ISi_1)_\omega$ and, in particular, contains $\Rfn(\ISi_1 + \Rfn(\ISi_1))$.
\end{proof}

Let us note that the non-isomorphism part of Proposition \ref{prop:isi1pra} is proved by isolating a natural algebraic property, which is relatively easy to formulate (however, we do not claim it to be definable in the language of bimodal diagonalizable algebras): the filter defined by the identity $[1]\theta = \top$ is generated by the elements of the form $\Box \psi \to \psi$ (i.e., there is some finite conjunction of such elements below any given element of this filter).

The above proof also goes through for the pair $S = \EA + \Rfn_{\Sigma_1}(\EA)$ and $T = \EA + \Rfn_{\Sigma_2}(\EA)$. The theories $S$ and $T$ are $\mathcal{B}(\Sigma_1)$-equivalent provably in $\EA$ (see \cite[Theorem 6]{Bek05}), whence they are $\mathcal{B}(\Sigma_1)$-coherent,
since $\EA$ is closed under $\Sigma_1$-collection rule by \cite[Lemma 4.1]{Bek98}. By the results of \cite{KolmBek19}, we have $\cC_S \equiv S + \Rfn(S)$ and $\cC_T \equiv T + \Rfn(T + \Rfn(T))$ (see \cite[Theorem 1]{KolmBek19}).
Note that the above argument relies on the fact that $\cC_S$ is generated by local reflection $\Rfn(S)$ (which is definable as a subset of the diagonalizable algebra), and it does not seem to work in case, when $\cC_S$ is generated by iterated reflection, e.g., $\Rfn(S + \Rfn(S))$, and not just $\Rfn(S)$.

In the following theorem and its corollary we show, however, that even when $\cC_S$ contains iterated local reflection it is still possible to obtain the analogue of Proposition \ref{prop:isi1pra} (or rather the example considered after the proposition). The corollary provides a family of natural pairs of theories with isomorphic diagonalizable algebras, but non-isomorphic bimodal diagonalizable algebras.

We still employ the idea of comparing the filters $\cC_S$ and $\cC_T$ in terms of the amount of local reflection contained in them. 
Now, however, the plan is to consider diagonalizable algebras 
of theories $\cC_T = C_T(T)$ and $\cC_S = C_S(S)$ 
instead of the bimodal algebras 
of $T$ and $S$, and to apply (a strengthened version of)  Shavrukov's non-isomorphism theorem, i.e., Theorem \ref{th:locrfn}. The first part of this plan is contained in the following lemma.

\begin{lemma}\label{lm:surjbimod}
Any epimorphism $(\cdot)^* \colon (\fL_T, \Box_T, [1]_T) \to (\fL_S, \Box_S, [1]_S)$ induces an epimorphism from $(\fL_{\cC_T}, \Box_T[1]_T)$ onto $(\fL_{\cC_S}, \Box_S[1]_S)$.
\end{lemma}
\begin{proof}
The Lindenbaum algebra $\fL_{\cC_T}$ is (isomorphic to) a quotient 
of $\fL_T$ by the filter $\cC_T$. Firstly, we show that the operation $\Box_T[1]_T$ is well-defined on $\fL_{\cC_T}$.
For all sentences $\phi$ and $\psi$ we have 
$$
\cC_T \vdash \phi \eqv \psi \Rightarrow T \vdash [1]_T \left(\phi \eqv \psi\right) \Rightarrow T \vdash \Box_T[1]_T \phi \eqv \Box_T[1]_T \psi \Rightarrow \cC_T \vdash \Box_T[1]_T \phi \eqv \Box_T[1]_T \psi,
$$
since $T \subseteq \cC_T$. The predicate $\Box_T[1]_T$ is a provability predicate for $\cC_T$, so $(\fL_{\cC_T}, \Box_T[1]_T)$ is a diagonalizable algebra of $\cC_T$, and likewise for $\cC_S$.

Since $(\cdot)^*$ is a homomorphism, for all sentences $\phi$ and $\psi$ we have
$$
\cC_T \vdash \phi \eqv \psi \Rightarrow T \vdash [1]_T
\left(\phi \eqv \psi\right) \Rightarrow S \vdash [1]_S\left(\phi^*
\eqv \psi^*\right) \Rightarrow \cC_S \vdash \phi^* \eqv \psi^*,
$$
so $(\cdot)^*$ is well-defined on $\fL_{\cC_T}$.
Since $S \subseteq \cC_S$ we also have that
$(\cdot)^*$ is an epimorphism from $\fL_{\cC_T}$ onto $\fL_{\cC_S}$, which preserves the composition of boxes $\Box[1]$. 
Thus, we obtain an epimorphism from
$(\fL_{\cC_T}, \Box_T[1]_T)$ onto  $(\fL_{\cC_S}, \Box_S[1]_S)$. 
\end{proof}

\begin{theorem}\label{th:bimodsurj}
Let $T$ and $S$ be such that $\cC_T \vdash \Rfn_{\Sigma_1}(\cC_S)$, provably in $\cC_T$.
Then there are no epimorphisms from $(\fD_T, [1]_T)$ onto $(\fD_S, [1]_S)$.
\end{theorem}

\begin{proof}
Lemma \ref{lm:surjbimod} implies that it is sufficient to show that there are no epimorphisms from $(\fL_{\cC_T}, \Box_T[1]_T)$ onto $(\fL_{\cC_S}, \Box_S[1]_S)$, which immediately follows from Theorem \ref{th:locrfn} and the assumption, since the predicate $\Box_T[1]_T$ is $\EA$-provably equivalent to a standard one, which can be proved exactly the same as in Lemma \ref{lm:boxkstand} (since $[1]_T$ also satisfies L{\"o}b's derivability conditions provably in $\EA$), and likewise for $S$.
\end{proof}

\begin{corollary}\label{cor:bimod}
	Let $T$ be $U + \Rfn_{\Sigma_2}(U)$ and $S$ be $U + \Rfn_{\Sigma_1}(U)$ for some theory $U$.
	Then there are no epimorphisms from $(\fD_T, [1]_T)$ onto $(\fD_S, [1]_S)$, while 
	$\fD_T \cong \fD_S$.
\end{corollary}
\begin{proof}
The isomorphism part follows from the fact that, just as in example above, $T$ and $S$ are $\mathcal{B}(\Sigma_1)$-equivalent provably in $\EA$ by \cite[Theorem 6]{Bek05}, whence $\fD_T \cong \fD_S$ by Theorem \ref{th:shaviso}. 

The other part is covered by Theorem \ref{th:bimodsurj} as follows.
We show that
\begin{equation*}\label{eq:ctrfncs}
\EA \vdash \forall \sigma \in \Sigma_1 \left(
\Box_T[1]_T \left(\Box_S[1]_S \sigma \to \sigma \right)
\right).
\tag{$\star$}
\end{equation*} 
Firstly, note that $\EA \vdash \forall \psi \left( \Box_S\psi \to \Box_T\psi\right)$ and
$$
\EA \vdash \forall \phi \left(\Box_S[1]_S\phi \to \Box_U[1]_S\phi\right).
$$
Indeed, formalize the following argument in $\EA$.
Assume $S \vdash [1]_S \phi$, i.e., $U + \Rfn_{\Sigma_1}(U) \vdash [1]_S\phi$. It follows that $U + \RFN_{\Sigma_1}(U) \vdash [1]_S \phi$, so
$U \vdash \neg [1]_U \bot \to [1]_S \phi$. But also $U \vdash [1]_U \bot \to [1]_S \phi$, since $\EA \vdash \forall \psi \left(\Box_U \psi \to \Box_S \psi\right)$. Thus, we have $U \vdash [1]_S$, as required.

Now,  we obtain \eqref{eq:ctrfncs} by formalizing the following argument.
Fix an arbitrary $\Sigma_1$-sentence~$\sigma$. In the following derivations we use the facts proved above and provable $\Sigma_2$-completeness of $[1]_T$,
\begin{align*}
U + \Rfn_{\Sigma_2}(U) \vdash 
\Box_S[1]_S \sigma
&\to \Box_U[1]_S\sigma\\
&\to [1]_S\sigma\\
&\to [1]_T\sigma\\
&\to [1]_T \left(\Box_S[1]_S \sigma \to \sigma\right),\\
\vdash 
\neg \Box_S[1]_S \sigma
&\to [1]_T\neg \Box_S [1]_S\sigma\\
&\to [1]_T \left(\Box_S[1]_S \sigma \to \sigma\right),
\end{align*}
whence $T \vdash [1]_T \left(\Box_S[1]_S \sigma \to \sigma\right)$, as required.
\end{proof}

As was discussed above, besides the example $S = \EA + \Rfn_{\Sigma_1}(\EA)$ and $T = \EA + \Rfn_{\Sigma_2}(\EA)$, where $\cC_S$ is generated by $\Rfn(S)$, Corollary \ref{cor:bimod} applies to the case when $\cC_S$ contains iterated reflection.
For instance, if we take $U = \ISi_1$, we have
$\cC_S \equiv \Rfn(S)_\omega$  and $\cC_T \equiv \Rfn(T)_{\omega + 1}$ (see \cite[Theorem 1]{KolmBek19}).

Finally, let us mention the following possible direction of research. Instead of enriching the signature of diagonalizable algebra we can consider its various substructures. Several fragments of diagonalizable algebra of a given theory have been studied in the literature. One example of that is the work  \cite{LindShav08} of Lindstr{\"o}m and Shavrukov  where the decidability of  $\forall^*\exists^*$-fragment of the lattice of $\Sigma_1$-sentences modulo $T$-provable equivalence is established. Another example is given by the recent study of the so-called relative provability degrees \cite{AndCai15}, where the authors study the properties of the lattice of all true $\Pi_2$-sentences modulo $T$-provable equivalence. Let us note that formulations of many results in \cite{AndCai15} and the methods used to obtain the proofs are computability-theoretic in nature and similar to those used for studying the degree structures in computability theory. This is mainly based on the fact that $\Pi_2$-sentences can be seen as asserting the totality of computable functions and vice versa. As has been noted by Shavrukov, some results of \cite{AndCai15} have quite natural proof-theoretic reformulations and can be established using known results and methods of proof theory and provability logic (e.g., \cite[Theorem 5.1]{AndCai15} is an immediate corollary of the G{\"o}del-Rosser theorem). Several operators acting on this lattice are introduced in this work, namely, skip $\phi \mapsto \Diamond_T \phi$, hop $\phi \mapsto \phi \land \Diamond_T\phi$ and jump $\phi \mapsto \la 1 \ra_T \phi$. Here we highlight one of the open questions formulated in that work (see \cite[Question 10.9]{AndCai15}).

\begin{question}
To what extent, if any, does the structure of the provability degrees depend  on  the  underlying  theory $T$?  Are  these  structures  isomorphic  for  different theories?
\end{question} 

\subsection*{Acknowledgments}
The author is grateful to Lev D.\,Beklemishev for useful discussions and to Vladimir Yu.\,Shavru\-kov for valuable suggestions and comments on an early draft of the paper. 

This work is supported by the Russian Science Foundation, in cooperation with the Austrian Science Fund, under grant RSF--FWF 20-41-05002.

\noindent 
Steklov Mathematical Institute of Russian Academy of Sciences,\\ Gubkina St. 8, 119991, Moscow, Russia\\
{\it E-mail:} \texttt{kolmakov.evgn@gmail.com}

\end{document}